\documentclass[12pt]{article}
\usepackage{fullpage}
\usepackage[utf8]{inputenc}
\usepackage{hyperref}
\usepackage{authblk}

\usepackage{amsmath, amssymb, mathtools, amsthm}
\usepackage{relsize,graphicx} 
\usepackage{fancyhdr}
\usepackage{bbm} 
\usepackage{verbatim}

\usepackage[T1]{fontenc}
\usepackage{graphicx}

\usepackage[dvipsnames]{xcolor}
\colorlet{dlav}{ForestGreen!300!}

\usepackage{mathpazo}

\usepackage{fourier} 
\usepackage[scaled=0.875]{helvet} 

\newtheorem{theorem}{Theorem}
\newtheorem{remark}{Remark}
\newtheorem{lemma}[theorem]{Lemma}

\numberwithin{equation}{section}
\numberwithin{theorem}{section}
\numberwithin{remark}{section}

\newcommand{\alphavar}{\tilde{\alpha}}

\newcommand{\C}{\mathlarger{\mathcal{C}}}

\newcommand{\F}{\mathcal{F}}

\newcommand{\B}{\mathcal{B}}

\newcommand{\Real}{\mathbb{R}}

\newcommand{\imply}{\Longrightarrow \quad }

\newcommand{\indi}[1]{\mathbbm{1}_\mathsmaller{#1}}

\newcommand{\expec}{\mathbb{E}}

\newcommand{\prob}{\mathbb{P}}

\newcommand{\varU}{\tilde{U}}

\newcommand{\renewalpt}{\left(0, -g/(1+\gamma) \right)}

\title{Inert drift system in a viscous fluid: Steady state asymptotics and exponential ergodicity}

\author{Sayan Banerjee\footnote{\texttt{sayan@email.unc.edu}} \ \ \ \ \ Brendan Brown\footnote{\texttt{bb@live.unc.edu}}}
\affil{University of North Carolina, Chapel Hill, USA}
\date{}							

\begin{document}
\maketitle
\begin{abstract}
We analyze a system of stochastic differential equations describing the joint motion of a massive (inert) particle in a viscous fluid in the presence of a gravitational field and a Brownian particle impinging on it from below, which transfers momentum proportional to the local time of collisions. We study the long-time fluctuations of the velocity of the inert particle and the gap between the two particles, and we show convergence in total variation to the stationary distribution is exponentially fast. We also produce matching upper and lower bounds on the tails of the stationary distribution and show how these bounds depend on the system parameters. A renewal structure for the process is established, which is the key technical tool in proving the mentioned results.\\


\end{abstract}

\tableofcontents
\section{Introduction and summary of results}
We study an inert drift system which models the joint motion of a massive (inert) particle and a Brownian particle in a viscous fluid in the presence of a gravitational field. The inert particle is impinged from below by the Brownian particle which transfers momentum to it proportional to the `local time' of collisions. This acceleration is countered by the viscosity of the fluid and the gravitational force acting on the inert particle. 

The interaction is non-standard because transfer of momentum is `non-Newtonian:' It can be thought of as an `infinite number of collisions' between the particles in any finite time interval, such that each collision results in an `infinitesimal momentum transfer.' This inert drift system is a simplified mathematical model for the motion of a semi-permeable membrane in a viscous fluid in the presence of gravity \cite{einstein,knight}. The membrane plays the role of the inert particle, which is permeable to the microscopic fluid molecules but not the macroscopic Brownian particle. The following system of stochastic differential equations characterizes the joint motion of particles in the present model:
\begin{equation}
\begin{cases}
dX_t = dB_t - dL_t\nonumber \\
dV_t = -(\gamma V_t + g) dt + dL_t\nonumber \\
dS_t = V_t dt\nonumber 
\end{cases}
\end{equation}
starting from $(X_0, V_0, S_0) \in \mathbb{R}^3$ with $S_0 \ge X_0$ and $V_0 \in (-g/\gamma, \infty)$. See \eqref{eqn:velocity_lower_bounded} for justification of the restricted range for $V_0$.
Here $\gamma, g > 0$ respectively denote the viscosity coefficient and the acceleration due to gravity, $B_t$ is a one-dimensional standard Brownian motion, $S_t$ is the trajectory of the inert particle, $V_t$ is the velocity of the inert particle and $L_t$ is the local time of collisions, which is defined as the unique continuous, non-negative, non-decreasing process such that $\int_0^t \indi{S_t - X_t = 0}dL_u = L_t$ and $S_t - X_t \geq 0 $ for all $t$. We will assume $S_0 = 0$ unless otherwise stated.

\cite{knight} initiated the study of inert drift systems by studying the case where $g=\gamma=0$. In this case, the system becomes transient, meaning the inert particle escapes. \cite{knight} determined the laws of the inverse velocity process and the `escape velocity.' Since then, numerous inert drift systems have been studied \cite{bass,white,BW}. \cite{Barnes2017} studied hydrodynamic limits for inert drift type particle systems. Moreover, stochastic differential equations somewhat similar in flavor to inert drift systems have recently appeared as diffusion limits of queuing networks like the join-the-shortest-queue discipline \cite{EG,queue,queue2}. \cite{grav} studied the inert drift model with $g>0, \gamma = 0$. Unlike \cite{knight} the model is recurrent, in other words the inert particle does not escape. The paper showed that the process hits a specified point almost surely and that the hitting times at this point has finite expectation. By decomposing the path into excursions between successive hitting times of such a point, the paper showed the process $(S_t-X_t, V_t)$ has a renewal structure and converges in total variation distance to a unique stationary distribution. Moreover, it was shown by solving an associated partial differential equation that the process $(S_t-X_t, V_t)$ has an explicit product form stationary distribution which is Exponential in the first coordinate and Gaussian in the second coordinate. Using this explicit form along with the renewal structure, the paper obtained sharp fluctuation estimates for the velocity $V_t$ and the gap between the particles $S_t - X_t$.

In this article, we analyze the full model, in which $g>0, \gamma>0$. In contrast with \cite{grav}, there is no explicit closed form for the joint stationary distribution of the velocity and the gap. Nevertheless, we establish a renewal structure analogous to \cite{grav} and thus obtain a tractable renewal theoretic representation of the stationary distribution (Theorem \ref{thm:statmeas_renewal}). By analyzing the excursions between successive renewal times, we obtain precise upper and lower bounds on the tails of the stationary distribution (Theorem \ref{thm:tails_under_stationarity}). Besides furnishing Exponential tails for the gap and Gaussian tails for the velocity, these bounds display the explicit dependence of the stationary distribution tails on the model parameters $g,\gamma$. We exploit renewal structure further to obtain parameter-dependent fluctuation estimates for the gap $S_t - X_t$ and the velocity $V_t$ (Theorem \ref{thm:fluctuations}), which also imply law of large numbers results for $S_t$ and $X_t$ (Theorem \ref{thm:lln}). 

One surprising aspect of the model arises from Theorem \ref{thm:fluctuations} which shows that $(S_t - X_t)/\log t$ is $O(\gamma/g)$ for large $t$ as compared to $O(1/g)$ for the model without viscosity studied in \cite{grav}. Thus, the fluctuation results and tail estimates for the case $\gamma=0$ cannot be recovered by taking $\gamma \rightarrow 0$ in our results. This shows that the \emph{qualitative behavior of the steady state changes on introducing viscosity}, and that the rare events contributing to tail estimates of the stationary measure arise in very different ways between the $\gamma=0$ and $\gamma>0$ cases (see Remark \ref{rare} after Theorem \ref{thm:fluctuations}).

In addition, we show that convergence to stationarity is exponentially fast, by appealing to Harris' Theorem (see \cite{hairandmat,meyntweedie} for versions of the theorem). The standard approach to Harris' Theorem relies on the existence of continuous densities for transition laws with respect to Lebesgue measure (see \cite{budhirajalee,cooke,mattinglypillai,cattiaux}) which, in turn, is shown by establishing that the generator of the process is hypoelliptic. However, the generator of our process $(S_t - X_t, V_t)$ is not hypoelliptic in the interior of the domain (when $S_t - X_t>0$) and the transition laws of $(S_t - X_t, V_t)$ do not have densities with respect to Lebesgue measure. This fact is in essence a consequence of the velocity's deterministic evolution when $S_t - X_t>0$. Therefore, possible ergodicity arises from non-trivial interactions between the drift, the Brownian motion and the boundary reflections. 

To circumvent these technical challenges, we once again use the renewal structure of our process to show a minorization condition \eqref{eqn:minorization_cond} and to obtain exponential moment estimates for hitting times to certain sets. Exponential moments provide a suitable Lyapunov function and thereby a drift condition \eqref{eqn:drift_cond} used to obtain exponential ergodicity via the techniques of \cite{meyntweedie,thor,ergodicity}.

Now we give an outline of the organization of the article. In Subsections \ref{renst} and \ref{mr2}, we describe the renewal structure of the system and state the main results of the article. In Section \ref{exun} we prove the existence and uniqueness of the process and, in particular, prove its strong Markov property and a Skorohod representation for the local time. In Section \ref{strep}, we obtain tail estimates on the distribution of the renewal time which, in particular, imply its integrability and the existence and uniqueness of the stationary measure. It also gives a tractable renewal theoretic representation of the stationary measure. In Section \ref{flucinter}, fluctuation bounds for the velocity and gap process between two successive renewal times are obtained.
These bounds imply tail estimates on the stationary measure and path fluctuation results, which are proved in Section \ref{osctail}. In Section \ref{geoerg}, we prove that the process converges to its stationary measure exponentially fast in total variation distance. Finally, Appendix \ref{hittimeest} is devoted to some technical estimates for hitting times which are used throughout the article.
\subsection{Notation}
We set the following notation:
\begin{eqnarray}
H_t &=& S_t - X_t \nonumber \\
\sigma(t) &=& \inf \{s \geq t\, | \,\, H_s = 0\}\nonumber \\
\tau_b^V &=& \inf \{s \geq 0\, | \,\, V_s = b\} \nonumber \\
\tau^V_B &=& \inf \{s \geq 0\, | \,\, V_s \in B\} \nonumber \\
\tau_x^H &=& \inf \{s \geq 0\, | \,\, H_s = x\} \nonumber \\
\tau_A &=& \inf \{s \geq 0\, | \,\, \left(H_s, V_s\right) \in A\} \nonumber \\
P^t\left( \left(h, \nu\right), \cdot\>\right) &=& \prob_{(h, \nu)}\left(\left(H_t, V_t\right) \in \cdot \>\right) = \prob\left(\left(H_t, V_t\right) \in \cdot \,\, | \,\, \left(H_0, V_0\right) = (h, \nu) \right) \nonumber\\
P^t\left(\left(h, \nu\right), \> f\right) &=& \expec_{\left(h, \nu\right)} f\left(H_t, V_t \right) \nonumber
\end{eqnarray}
where $f$ is a measurable function such that $\expec_{\left(h, \nu\right)} f\left(H_t, V_t \right)$ is defined, $b \in \mathbb{R}$, $x \ge 0$ and $B \subset \mathbb{R}, A \subset \mathbb{R}^2$ are Borel-measurable sets. Unless otherwise stated, $S_0 = X_0 = B_0 = 0$. The notation $|\cdot|$ will be used for the Euclidean norm on $\Real$ or $\Real^2$, the space being clear from context. We will work with system equations reformulated as
\begin{equation}\label{eqn:system}
\begin{cases}
dH_t = V_t\,dt - dB_t + dL_t \\
dV_t = -(\gamma V_t + g)\,dt + dL_t
\end{cases}
\end{equation}
for initial conditions $(H_0, V_0) = (h,\nu) \in \mathbb{R}_+ \times \left(-\frac{g}{\gamma}, \infty\right)$ and local time $L$ defined as the unique continuous, non-negative, non-decreasing process such that $\int_0^t \indi{H_t = 0}dL_u = L_t$ and $H_t \geq 0 $ for all $t$.
Sometimes, we will write $L^{(h, \nu)}_t$ in place of $L_t$ to elucidate dependence on the initial conditions ${(h, \nu)}$. We will show in Theorem \ref{thm:existence} that the Skorohod representation for the local time is valid, namely
\begin{equation}\label{sm}
L_t = \sup_{u\leq t}\left(-h + B_u-\int_0^uV_wdw\right) \vee 0.
\end{equation}
We assume $C' > 0$ are fixed throughout.  We will write the state-space of the solution to \eqref{eqn:system} as $S = \Real_+ \times \left(-\frac{g}{\gamma}, \infty\right)$, but also use $S$ for the state-space of the more general diffusion \eqref{eqn:system2} when there is no danger of confusion.

$K, K'$ etc. will always denote positive constants, not depending on $\gamma, g$. $c, c', C, C'$ etc. will denote positive constants dependent on $\gamma, g$. Values of constants might change from equation to equation without mention. Throughout, $\gamma, g > 0$ are fixed.

\subsection{System properties}\label{section:properties}
We adapt general techniques from stochastic differential equations to our context to show that a strong solution to \eqref{eqn:system} exists, is pathwise unique and has the strong Markov property, proven in Theorem \ref{thm:existence}.

We state a few fundamental properties of the system's motion that are integral to all results in this paper. For initial conditions $(H_0, V_0) = (h, \nu) \in S$, since $L_t$ is non-negative for each $t \ge 0$, \eqref{eqn:system} shows
\begin{equation} \label{eqn:velocity_lower_bounded}
    V_t \geq \left(v + g/\gamma\right)e^{-\gamma t} - g/\gamma > -g/\gamma \quad \forall \, t \geq 0,
\end{equation}
which justifies the state-space $S$ rather than $\Real_+ \times \Real$. Moreover, for $t < \sigma(0)$,
\begin{eqnarray} \label{eqn:system_starts_interior}
V_t &=& \nu -\int_0^t (\gamma V_s + g) \, ds = \nu -\gamma S_t - gt = \left(\nu + g/\gamma\right)e^{-\gamma t} - g/\gamma,\nonumber \\
H_t &=& h + S_t - B_t = h + \nu/\gamma - V_t/\gamma - B_t - t g/\gamma  \nonumber \\
&\leq&h + \nu/\gamma + g/\gamma^2 - B_t - tg/\gamma
\end{eqnarray}
where the last inequality follows from $V_t \ge -g/\gamma$ for all $t \ge 0$.
Thus $H$ is dominated by Brownian motion with drift $-g/\gamma$ in $S^\circ$. The last equality in the velocity equation in \eqref{eqn:system_starts_interior} comes from solving the ODE for $V$ obtained from \eqref{eqn:system} without $L_t$, which is zero for $t \in [0, \sigma(0))$. \eqref{eqn:system_starts_interior} shows the velocity increases on $\partial S$ and only there, i.e. at $t$ when $H_t = 0$. Otherwise stated, if $(H_0, V_0) \in S^\circ$,  
\begin{equation} \label{eqn:increase_on_boundary_only}
    V_t \text{  is decreasing on  } \quad t \in \left[0, \, \sigma(0)\right) \quad \text{and} \quad  H_{\tau_b^V} = 0  \quad \text{  for   } b > V_0.
\end{equation}
Properties \eqref{eqn:system_starts_interior} and \eqref{eqn:increase_on_boundary_only} apply equally to the process started from some arbitrary time after corresponding changes in the stopping times.

We conclude this section with an implication of \eqref{eqn:velocity_lower_bounded} and \eqref{eqn:system_starts_interior} that we often will use: \eqref{eqn:velocity_lower_bounded} shows $V_t$ cannot hit a level $a \in (-g/\gamma, V_0)$ before $\left(V_0 + g/\gamma\right)e^{-\gamma t} - g/\gamma$ does so at $\frac{1}{\gamma}\log\left(\frac{V_0 + g/\gamma}{a + g/\gamma} \right)$. When $H_0 > 0$ and $\frac{1}{\gamma}\log\left(\frac{V_0 + g/\gamma}{a + g/\gamma} \right) \leq  \sigma(0)$, \eqref{eqn:system_starts_interior} shows $V$ hits $a$ exactly at $\frac{1}{\gamma}\log\left(\frac{V_0 + g/\gamma}{a + g/\gamma} \right)$. In summary, for $V_0 > a$,
\begin{equation} \label{eqn:hitting_time_velocity}
    \tau_a^{V} \geq \frac{1}{\gamma}\log \left(\frac{V_0 + g/\gamma}{a + g/\gamma} \right), \quad \tau_a^{V} = \frac{1}{\gamma}\log \left(\frac{V_0 + g/\gamma}{a + g/\gamma} \right) \>\> \text{for} \> \> H_0 > 0, \> \sigma(0) \geq \frac{1}{\gamma}\log \left(\frac{V_0 + g/\gamma}{a + g/\gamma} \right).
\end{equation}
We note this also implies $\left\{\tau_a^V \leq \sigma(0) \right\} = \left\{\tau_a^V \leq \sigma(0), \> \tau_a^V =  \frac{1}{\gamma}\log\left(\frac{V_0 + g/\gamma}{a + g/\gamma} \right)\right\}$ for $H_0 > 0, V_0 > a$.
\subsection{Main results: Stability and renewal theory}\label{renst}
At the heart of our proofs is the renewal structure of the process, which we now formalize. For any stopping time $\alpha$, define $\tau_b^V(\alpha) = \inf \left\{t \geq \alpha \,\big|\> V_t = b\right\}$, so that $\tau_b^V = \tau_b^V(0)$. Define a sequence of stopping times where the process $(H_t, V_t)$ visits the point $\renewalpt$, which we call the renewal point, as follows: Set $a = -(g + g/2\gamma)/(1+\gamma)$ and $b = -(g - g/2(1+\gamma))/(1+\gamma)$. Define $\zeta_{-1}=0$ and
\begin{eqnarray} \label{eqn:renewal_def}
\zeta &\coloneqq& \zeta_0 = \inf \left\{t \geq \tau_a^V\wedge\tau_b^V \,\big|\> \left( H_t, V_t\right) = \renewalpt\right\}, \nonumber \\
\zeta_{n+1} &\coloneqq& \inf \left\{t \geq \tau_a^V(\zeta_n)\wedge\tau_b^V(\zeta_n) \,\big|\> \left( H_t, V_t\right) = \renewalpt\right\}  \quad \quad n \geq 0.
\end{eqnarray}
The choices of $a$ and $b$ are not too important as long as $a \in (-g/\gamma, -g/(1+\gamma))$ and $b \in (-g/(1+\gamma),0)$, and their values are chosen as above for computational convenience.

We follow an approach similar to \cite{grav} and \cite{queue} in showing that $\zeta$ (and hence each $\zeta_j$) is integrable by decomposing the path between carefully chosen hitting times. 
\begin{theorem} \label{thm:renewal} There exist constants $t_0(\gamma, g), c > 0$ such that
$$\prob_{\mathsmaller{\renewalpt}}\left(\zeta > t^2\right)  \leq e^{-c\, t}$$
for all $t > t_0(\gamma, g)$.
\end{theorem}
Since each $\zeta_j$ is integrable and the process is strong Markov, we have
\begin{equation} \label{eqn:renewal_sequence}
    \left\{\zeta_{j+1} - \zeta_{j}\right\}_{j\geq0} \> \text{  and  } \> \left\{\left(H_t, V_t\right)_{t \in \left[\zeta_j, \zeta_{j+1}\right)} \right\}_{j\geq0} \quad \text{ are i.i.d. }
\end{equation}
The existence and a representation of the stationary distribution, along with an ergodicity result for time averages, follows from the integrability of $\zeta$ and the observation \eqref{eqn:renewal_sequence}  using the techniques developed in Ch. 10 of \cite{thor}.
\begin{theorem} \label{thm:statmeas_renewal} 
The solution to equation \eqref{eqn:system} has a unique stationary distribution $\pi$. For $A \in \B(S)$,
\begin{equation*} 
\pi(A) = \frac{\expec_{\mathsmaller{\renewalpt}}\left(\int_0^\zeta \, \indi{ (H_t, V_t) \in A} \, dt\right) }{\expec_{\mathsmaller{\renewalpt}}\>\left(\zeta \right)}.
\end{equation*}
In addition, for all $f: S \mapsto \Real$ bounded and measurable, we have
\begin{equation*} 
\lim_{t \to \infty} \> \frac{1}{t} \int_0^t P^s((h, \nu), f) \> ds = \int f(y) \, \pi(dy).
\end{equation*}
\end{theorem}
\subsection{Main results: Path fluctuations, tail estimates for stationary measure and exponential ergodicity}\label{mr2}
The renewal structure laid out in Subsection \ref{renst} allows us to make statements about the long-time behavior of the system by studying its behavior in the (random) time interval $[0, \zeta)$, starting from the renewal point $\renewalpt$. An analysis of fluctuations in this random time interval translates to tail estimates for the stationary distribution $\pi$ as displayed in Theorem \ref{thm:tails_under_stationarity}. It also produces long-time oscillation estimates for $V_t$ snd $H_t$, given in Theorem \ref{thm:fluctuations}.
\begin{theorem} \label{thm:tails_under_stationarity} There exists positive constants $y'(\gamma, g), c', c$ such that for all $y > y'(\gamma, g)$,
\begin{equation*}
    e^{-c'} \> e^{-4 \, (1+\gamma)\left(y + \frac{g}{1+\gamma}\right)^2} \leq \pi\left(\Real_+ \times (y, \infty) \right) \leq  \>e^{c} \>   \>e^{-\frac{1+\gamma}{8}\left(y + \frac{g}{1+\gamma}\right)^2}.
\end{equation*}
Also, there exists positive constants $x'(\gamma, g), C', C$ such that for all $x > x'(\gamma, g)$,
\begin{equation*}
    e^{-C'} e^{- \frac{4 g x}{\gamma}} \leq \pi\left((x, \infty) \times \Real\right) \leq e^{C} \,\, e^{- \frac{g x}{32 \gamma}}.
\end{equation*}
\end{theorem}
\begin{theorem} \label{thm:fluctuations} For each $(h, \nu) \in S$, $\prob_{(h, \nu)}$-a.s., 
\begin{eqnarray}
\frac{1}{\sqrt{2}\sqrt{1+\gamma}} \leq &\limsup_{t \to \infty} \frac{V_t}{\sqrt{\log t}}& \leq 2 \frac{1}{\sqrt{1+\gamma}} \nonumber \\
\frac{\gamma}{2g} \leq &\limsup_{t \to \infty} \frac{H_t}{\log t}& \leq 16 \frac{\gamma}{g}. \nonumber 
\end{eqnarray}
\end{theorem}

\begin{remark}\label{rare}
Note that one cannot recover the analogous results in \cite{grav} (see Theorems 2.1 and 2.2 in \cite{grav}) by taking $\gamma \rightarrow 0$ in the results for the gap displayed in Theorems \ref{thm:tails_under_stationarity} and \ref{thm:fluctuations} above. This is because, when $\gamma=0$, the primary contribution to the tail estimates for the stationary gap distribution comes from rare events when a large upward climb of the Brownian particle leads to a large positive value for the velocity of the inert particle. During such events the inert particle moves up rapidly, and the Brownian particle cannot `keep up,' resulting in a large gap. 

In contrast when $\gamma>0$, the viscosity term $-\gamma V$ ensures the inert particle moves more `sluggishly' and cannot escape the Brownian particle. Thus, the large gaps arise when the Brownian particle escapes the inert particle by having a large downward fall. As we will show, the gap behaves like reflected Brownian motion with drift $-g/\gamma$ during such excursions. 
\end{remark}

The oscillation estimates obtained in Theorem \ref{thm:fluctuations} lead to a law of large numbers result for the trajectories $S_t$ and $X_t$.

\begin{theorem} \label{thm:lln} For each $(h, \nu) \in S$, $\prob_{(h, \nu)}$-a.s., 
\begin{equation*}
    \lim_{t \to \infty} \frac{S_t}{t} = \lim_{t \to \infty} \frac{X_t}{t} = -\frac{g}{1+\gamma}.
\end{equation*}
\end{theorem}
The next result shows that convergence to stationarity happens exponentially fast. Define the total variation norm $\| \cdot \|_{TV}$ for signed measures $\mu$ as $\|\mu\|_{TV} = \sup_{A \in \B(S)}|\mu|(A)$.
\begin{theorem} \label{thm:expo_ergodic} There exists a function $G:S \mapsto [1, \infty)$ and constants $D \in \left(0, \infty \right)$, $\lambda \in \left(0, 1 \right)$ such that for all $(h, \nu) \in S$ and $t \geq 0$,
$$\|P^t\left((h, \nu), \cdot\>\right) - \pi \|_{TV} \leq G(h, \nu)\, D \lambda^t.$$
\end{theorem}


\section{Existence and uniqueness of the process}\label{exun}

We show the existence of a pathwise unique strong solution to equations \eqref{eqn:system} that is also a strong Markov process. The results proved here are for systems slightly more general than our current model. Consider the system of equations for $U = (U^1, U^2)$, with $U_0 = u = (u^1, u^2) \in \Real_+ \times \Real$ fixed,
\begin{align} \label{eqn:system2}
    dU^1_t &= \varphi(U_t)\,dt - dB_t + dL^{u}_t \\
    dU^2_t &= \phi(U_t) dt + dL^{u}_t, \nonumber 
\end{align}
where $L^{u}_t$ is a continuous, non-negative, non-decreasing process such that $\int_0^t \indi{U^1_s = 0}dL^{u}_s = L^{u}_t$ and $U^1_t \geq 0$ for all $t$. A solution to \eqref{eqn:system2} is therefore in $\C(\Real_+, S)$, the space of continuous functions on $\Real_+$ taking values in $S = \Real_+ \times \Real$. We write $\partial S = \{(u^1, u^2)\,: \,\, u^1 = 0 \}$ .

There has been substantial previous work in this direction. \cite{ikedawatanabe} establishes the existence of weak solutions and uniqueness in law for very general reflected systems. Existence and uniqueness results for inert drift systems have been addressed in \cite{knight} and \cite{white} for models where the velocity is proportional to the local time and consequently is an increasing process. \cite{bass} deals with weak existence and stationarity of inert drift systems on bounded domains. Our model differs qualitatively from those works, and we prove the existence of a pathwise unique strong solution to equations \eqref{eqn:system} that is also a strong Markov process. Under regularity conditions imposed on the drift $(\varphi,\phi)$, our proof adapts the standard procedure for existence and uniqueness of stochastic differential equations without reflection, with appropriate modifications to incorporate the reflection term.

We use the standard norm on $\C(\Real_+, S)$, the continuous functions on the half-line taking values in $\Real^d$, is given by $|f| = \sum_{n=1}^\infty \frac{1\wedge|f|_n}{2^n}$, where $|\cdot|_n$ is the supremum norm on  $\C([0, n], S)$. Also recall that a transition semigroup (Definition 6.1 in \cite{legall}) is a real-valued map $(t, u, f) \mapsto P^t(u, f)$ for $t \geq 0$, $u \in S$ and $f$ bounded, measurable real-valued functions on $S$ such that $A \mapsto P^t(u, \indi{A})$ is a probability measure on the Borel $\sigma$-field of $S$ for each $t$ and $u$, $\left(P^t \circ P^s\right)(u, f) = P^{t+s}(u, f)$ for each $s, t \geq 0$, $P^0(u, f) = f(u)$ for all $u$ and $(t, u) \mapsto P^t(u, f)$ is measurable for each such $f$. 

\begin{theorem}\label{thm:existence} Assume $\phi, \varphi$ are globally Lipschitz. Then for each initial condition $u = (u^1, u^2) \in S$ there exists a pathwise unique, strong solution to equations \eqref{eqn:system2} in $\C(\Real_+, S)$. The solution is a strong Markov process, and we have the Skorohod representation $$L^{u}_t = \underset{\ell\leq t}{\sup}\left(-u^1 + B_\ell - \int_0^\ell \varphi(U_s) \, ds\right)\vee 0.$$

Write $U(u)$ for the solution with $U_0 = u \in S$. The solution $U$ may be chosen such that $u \mapsto U(u)$ is continuous in the topology of $\C(\Real_+, S)$. $P^t\left(u, f\right) = \expec_u\left(f\left(U_t\right) \right)$, where $f:S \mapsto \Real$ is bounded and measurable, defines a transition semigroup.
\end{theorem}

\begin{proof}
We suppress the notation for initial conditions when the distinction is unnecessary. Define the stopping time and stopped processes
\begin{equation}\label{eqn:localisation}
    T_N = \inf\left\{t \geq 0\,:\,\, |U_t| > N \right\}, \quad \quad  U^N_t = U_{t\wedge T_N}, \quad \quad \> N \geq 1.
\end{equation}
The drift vector and diffusion matrices of the stopped system satisfy the boundedness and Lipschitz conditions required in Theorem 7.2, Chapter IV.7 of \cite{ikedawatanabe}, which shows that for each $N \geq 1$ there exists a weak solution $U^N$ to \eqref{eqn:system2} in $\C(\Real_+, S)$, which has the strong Markov property and is unique in law. Skorohod's lemma (Lemma 6.14 \cite{kar}) gives the representation for $L^u$ stated in the theorem.

Uniqueness is in fact pathwise: Take $U^N, (U')^N$ to be two weak solutions, defining $T'_N$ analogously to \eqref{eqn:localisation}. We can assume that $U^N, (U')^N$ exist on a single filtered probability space and that the Brownian motions corresponding to the solutions are the same. Such a construction relies on regular conditional probabilities as described in \cite{kar} Chapter 5D.

We use Gronwall's lemma in the typical way to show pathwise uniqueness. Set $\varU_t = U_{t\wedge T_N \wedge T'_N} - U'_{t\wedge T_N \wedge T'_N} $, suppressing the superscript $N$ on $\varU$. Denote the local times corresponding to $U^N, (U')^N$ as $L, L'$. Fixing $T > 0$, we calculate for all $t \in [0, T]$
\begin{multline} \label{eqn:gronwall_bound_uniqueness}
    \expec\left(|\varU_t|^2 \right) \leq \expec\left(\underset{s \leq t}{\sup} \left|\varU_s\right|^2 \right) \leq K\left(\expec\>\underset{s \leq t}{\sup}\left(L_{s\wedge T_N \wedge T'_N} - L'_{s\wedge T_N \wedge T'_N} \right)^2\right) \\
    + K \> \expec\left(\underset{s \leq t}{\sup}\left|\int_0^{s\wedge T_N \wedge T'_N} \varphi(U^N_\ell) - \varphi((U')^N_\ell) \, d\ell \right|^2 + \underset{s \leq t}{\sup}\left|\int_0^{s\wedge T_N \wedge T'_N} \phi(U^N_\ell) - \phi((U')^N_\ell) \, d\ell \right|^2\right) \\
    \leq K' \> \expec\left(\underset{s \leq t}{\sup}\left|\int_0^{s\wedge T_N \wedge T'_N} \varphi(U^N_\ell) - \varphi((U')^N_\ell) \, d\ell \right|^2 + \underset{s \leq t}{\sup}\left|\int_0^{s\wedge T_N \wedge T'_N} \phi(U^N_\ell) - \phi((U')^N_\ell) \, d\ell \right|^2\right)\\
    \leq K'' T\>\expec \left(\int_0^t |\varU_s|^2 \, ds \right).
\end{multline}
$K, K', K''>0$ are constants not depending on $T$, $u$ or $N$. The third inequality follows from the explicit form of $L$ furnished by Skorohod's lemma, and the last from using Jensen's inequality and the Lipschitz property. By Gronwall's lemma, $\varU_t = 0$ on $[0, T]$. Since $T$ was arbitrary, $U_{t\wedge T_N \wedge T'_N} = U'_{t\wedge T_N \wedge T'_N} $ for all $t$ and each $N$. Calculations almost identical to \eqref{eqn:gronwall_bound_uniqueness} with $U^N$ and $(U')^N$ in place of $\tilde{U}$, along with Gronwall's lemma again, imply $\lim_{N \to \infty}\> T_N = \infty$ and the same for $T'_N$. If $N_1 \leq N_2$, then $U_t^{N_1} = U_t^{N_2}$ for $t \leq T_{N_1}$, and we define a continuous process $U$ for all time $t$ such that $U_t = U^N_t$ for any $N$ such that $t \leq T_N$. Similarly define $U'$ for all time. Taking $N \to \infty$, we have $U_t = U'_t$ for all $t$, i.e. pathwise uniqueness holds.

Arguments based on regular conditional distributions used to prove the Yamada-Watanabe Theorem, and its Corollary 3.23, Chapter 5D of \cite{kar}, remain valid in the setting of Theorem \ref{thm:existence} once weak existence and pathwise uniqueness are shown. Therefore $U$ is the unique strong solution to \eqref{eqn:system2}. 

To show $u \mapsto U(u)$ is continuous, we can follow almost exactly the argument as in the proof of Theorem 8.5 in \cite{legall}. A minor change is needed because of the term $L^{u}_t$: consider the solutions $U(u)$ and $U(\overset{\sim}{u})$ starting from distinct points $u$ and $\overset{\sim}{u}$. It suffices to show, 
\begin{equation} \label{eqn:gronwall_bound_cont}
    \expec\left(\underset{s \leq t}{\sup} \left|U^N(u)_s - U^N(\overset{\sim}{u})_s  \right|^2 \right) \leq K\left(|u - \overset{\sim}{u}|^2 + T^2 \int_0^t \expec\left(\left|U^N(u)_s - U^N(\overset{\sim}{u})_s  \right|^2 \right) \, ds\right) 
\end{equation}
for any $t \in [0, T]$, for all $T > 0$ and $N \geq 1$. Then, as in Theorem 8.5 of \cite{legall}, continuity of $u \mapsto U(u)$ follows by sending $N\to \infty$ and using \eqref{eqn:gronwall_bound_cont} to apply Gronwall's and Kolmogorov's continuity lemmas (Theorem 2.9 of \cite{legall}). However, \eqref{eqn:gronwall_bound_cont} follows exactly as in the proof of \eqref{eqn:gronwall_bound_uniqueness} by replacing $L, L'$ with $L^u, L^{\overset{\sim}{u}}$ and $\varU$ with $U^N(u) - U^N(\overset{\sim}{u})$ for any $u, \overset{\sim}{u} \in S$. The strong Markov property for $U$ holds since it does for $U^N$, by the calculation 
$$\indi{T_N > \tau + t} \expec\left( f\left(U_{\tau + t}(u)\right) \> \big| \> \F_{\tau}\right) = \indi{T_N > \tau + t} \expec\left( f\left(U_{\tau\wedge T_N + t}(u)\right) \> \big| \> \F_{\tau\wedge T_N}\right) = \indi{T_N > \tau + t} P^t\left(U_{\tau\wedge T_N}, f \right),$$
for any finite stopping time $\tau$ and $t > 0$, and any bounded continuous $f$. As $N \to \infty$, $\indi{T_N > \tau + t} \to 1$ and the right-hand side tends to $P^t\left(U_{\tau}, f \right)$ by continuity. To check that $P^t(u, f)$ defines a transition semigroup, it remains only to show $t \mapsto P^t(u, f)$ is measurable for every $u$ and bounded measurable $f$. However, the Dominated Convergence Theorem implies $t \mapsto P^t(u, f)$ is continuous for every bounded continuous $f$, and the proof is completed by a standard Monotone Class Theorem argument.
\end{proof}

\section{Existence, uniqueness and representation of the stationary measure}\label{strep}

This section is devoted to proving Theorems \ref{thm:renewal} and \ref{thm:statmeas_renewal}. We consider two cases. The case when the renewal point is approached from below, that is when the velocity hits the level $a< -g/(1+\gamma)$ before $b>-g/(1+\gamma)$, is considered in Subsection \ref{sec:renewal_from_below}. Note that in this case, the first hitting time of the level $-g/(1+\gamma)$ by the velocity after hitting $a$ corresponds to the renewal time, by \eqref{eqn:increase_on_boundary_only}. The other case when $b$ is hit before $a$ is considered in Subsection \ref{sec:renewal_from_above}. In this case, the velocity can hit level $-g/(1+\gamma)$ after hitting $b$ without the gap being zero at this hitting time. Thus, the velocity can fall below $-g/(1+\gamma)$ after hitting $b$, then approach $-g/(1+\gamma)$ from below for the renewal time to be attained. Integrability of the renewal time is proven by splitting the path into excursions between carefully chosen stopping times and estimating the duration of each such excursion.

\subsection{Renewal point approached from below} \label{sec:renewal_from_below}

Fix $a = -\left(g + g/2\gamma \right)/(1+\gamma) \in \left(-g/\gamma, -g/(1+\gamma) \right)$, and $b = -\left(g - g/2(1+\gamma) \right)/(1+\gamma) \in \left(-g/(1+\gamma),  0\right).$ Define $\alpha_{-1} = 0$ and $\alpha_0 =  \tau_a^V \wedge \tau^V_b$. If $\tau_b^V < \tau_a^V$, define $\alpha_j = \alpha_0$ for all $k \ge 0$ and $N^{-} = 0$. For $k \ge 0$, if $V_{\alpha_{3k}} = a$, define
\begin{eqnarray} \label{eqn:def_stop_times_below}
\alpha_{3k + 1} &=& \sigma(\alpha_{3k}) \nonumber \\
\alpha_{3k + 2} &=& \inf\left\{t \geq \alpha_{3k + 1} \,|\,\, V_t = -\left(g + g/4\gamma \right)/(1+\gamma)\right\} \nonumber \\
\alpha_{3k + 3} &=& \inf\left\{t \geq \alpha_{3k + 2} \,|\,\, V_t = -g/(1+\gamma) \quad \text{or} \quad a\right\} \nonumber\\
N^- &=& \inf\left\{k \geq 1 \, | \,\, V_{\alpha_{3k}} = -g/(1+\gamma) \right\}.
\end{eqnarray}
If $V_{\alpha_{3k}} = -g/(1+\gamma)$, then $\alpha_j = \alpha_{3k} \ \forall \, j \geq 3k$. Unless otherwise stated, we assume $\left(H_0, V_0\right) = \renewalpt$.  Lemma \ref{lemma:vel_escape_interval} says $\tau_a^V\wedge \tau_b^V < \infty $ a.s. with respect to the measure $\prob_{\mathsmaller{\renewalpt}}$. Lemma \ref{lemma:mgf_hitting_boundary} shows $ \alpha_1 < \infty$, and the proof of Lemma \ref{lemma:renewal_above_time_diffs} applied successively to $\alpha_2 - \alpha_1$, $\alpha_3 - \alpha_2$ etc. will show the remaining $\alpha_j$ are finite as well.

\begin{lemma} \label{lemma:stop_renewal_from_below} For all $k\geq 0$,
$$\prob_{(0, -\frac{g}{1+\gamma})}\left(N^- > k\right) \leq  \left( 1 - \frac{1}{\sqrt{2\pi}}\frac{\frac{g}{\sqrt{2\gamma}(1+\gamma)}}{\left(\frac{g}{\sqrt{2\gamma}(1+\gamma)}\right)^2 + 1}\, \exp\left\{-\frac{g^2}{4\gamma(1+\gamma)^2}\right\}\right)^k .$$
\end{lemma}

\begin{proof}
Write $\epsilon = g/(2\gamma)$. For $\left(H_0, V_0\right) = \left(0, -\frac{g + \epsilon/2}{1+\gamma}\right)$ and $t < \tau_{- \frac{g + \epsilon}{1+\gamma}}^V\wedge \tau_{- \frac{g}{1+\gamma}}^V$, define
$$
Y_t := -\frac{g+\epsilon/2}{1+\gamma} - t\frac{g}{1+\gamma} + \underset{s \leq t}{\sup}\left(B_s + s \frac{g}{1+\gamma} \right).
$$
As $L_t \ge \underset{s \leq t}{\sup}\left(B_s + s \frac{g}{1+\gamma} \right)$ and $V_t \le -g/(1+ \gamma)$, therefore
$$
V_t = V_0 -\gamma \int_0^tV_udu - gt + L_t \ge Y_t
$$
for all $t < \tau_{- \frac{g + \epsilon}{1+\gamma}}^V\wedge \tau_{- \frac{g}{1+\gamma}}^V$.
Note that $Y_t$ cannot hit $- \frac{g + \epsilon}{1+\gamma}$ before time $\frac{\epsilon}{2g}$. If  $\underset{s \leq \epsilon/(4g)}{\sup}\left(B_s + s \frac{g}{1+\gamma} \right)\ge \frac{3}{4} \frac{\epsilon}{1+\gamma}$, then $Y_t$ and thus $V_t$ must have hit $- \frac{g}{1+\gamma}$ by time $\epsilon/(4g)$. Thus,
\begin{multline*}
\prob_{(0, -\frac{g}{1+\gamma})}\left(N^- > 1 \right) \le \prob_{\left(0, -\frac{g + \epsilon/2}{1+\gamma}\right)}\left( \tau_{- \frac{g + \epsilon}{1+\gamma}}^V < \tau_{- \frac{g}{1+\gamma}}^V\right)
\le \prob\left(Y_t \text{ hits } -\frac{g + \epsilon}{1+\gamma} \text{ before } -\frac{g}{1+\gamma} \right)
\\ \le \prob\left(\underset{s \leq \frac{\epsilon}{4g}}{\sup}\left(B_s + s \frac{g}{1+\gamma}\right) < \frac{3}{4}\frac{\epsilon}{1+\gamma}\right) \le \prob\left(B_{\epsilon / 4g} < \frac{1}{2}\frac{\epsilon}{1+\gamma} \right) \leq 1 - \frac{1}{\sqrt{2\pi}}\frac{\frac{\sqrt{\epsilon g}}{1+\gamma}}{\left(\frac{\sqrt{\epsilon g}}{1+\gamma}\right)^2 + 1}\, e^{-\frac{g \epsilon}{2(1+\gamma)^2}},
\end{multline*}
where the first inequality follows from the definition of $N^- $ and the last inequality follows from a standard lower bound on the normal distribution function. Successive application of the strong Markov property yields the result.
\end{proof}

\begin{lemma} \label{lemma:renewal_above_time_diffs}
Fix $\gamma, g > 0$. There exist positive constants $t_0(\gamma, g), c$ such that for $j = 1, 2, 3$,$k \geq 0$
\begin{equation*} 
    \prob_{\mathsmaller{\renewalpt}}\left(\alpha_{3k + j} - \alpha_{3k + j - 1} > t \right) = \prob_{\mathsmaller{\renewalpt}}\left(\alpha_{3k + j} - \alpha_{3k + j - 1} > t , \,\, N^- \geq k+1\right) \leq e^{-c\, t},
\end{equation*}
for all $t > t_0(\gamma, g)$.
\end{lemma}
\begin{proof} By definition, if $\tau_a^V < \tau_b^V$ then $\alpha_0 = \tau_a^V$ and $\alpha_1 = \sigma(\tau_a^V)$. Lemma \ref{lemma:gap_hits_zero_before_long} shows there exists a positive constant $c$ such that for $t$ sufficiently large,
\begin{equation}\label{eqn:0renewal_above_time_diffs}
    \prob_{\mathsmaller{\renewalpt}}\left(\alpha_{1} - \alpha_{0} > t\right) = \prob_{\mathsmaller{\renewalpt}}\left(\sigma(\tau_a^V) - \tau_a^V > t, \> \tau_a^V < \tau_b^V \right) \leq e^{-ct}.
\end{equation}
Now consider the difference $\alpha_{3k + 1} - \alpha_{3k}$ for $k \geq 1$. If $N^- \leq k$ then $\alpha_{3k + 1} - \alpha_{3k} = 0$. When $N^- \ge k + 1$, $\alpha_{3k - 1}$ is a point of increase of the velocity, since it has hit $-\frac{g+ \frac{g}{4\gamma}}{1 + \gamma}$ from below. By \eqref{eqn:increase_on_boundary_only}, $H_{\alpha_{3k - 1}} = 0$. $\alpha_{3k}$ is the first time after $\alpha_{3k - 1}$ that $V$ hits $a$, and $\alpha_{3k + 1} = \sigma(\alpha_k)$. In addition, $N^- \geq k+1$ implies the velocity starting from time $\alpha_{3k-1}$ does not hit $-g/(1+\gamma)$ before $a$ and therefore does not hit $b > -g/(1+\gamma)$ before $a$ either. Since $\left(H_{\alpha_{3k - 1}}, V_{\alpha_{3k-1}}\right) \in \{0\} \times (a, b)$, we apply the strong Markov property at $\alpha_{3k - 1}$ and Lemma \ref{lemma:gap_hits_zero_before_long} to show 
\begin{multline}\label{eqn:1renewal_above_time_diffs}
    \prob_{\mathsmaller{\renewalpt}} \left(\alpha_{3k + 1} - \alpha_{3k} > t\right) = \prob_{\mathsmaller{\renewalpt}} \left(\alpha_{3k + 1} - \alpha_{3k} > t,\> N^- \geq k + 1\right)\\
    = \expec_{\mathsmaller{\renewalpt}} \left(\mathbbm{1}_{N^- \geq k} \> \prob_{\left(H_{\alpha_{3k - 1}}, V_{\alpha_{3k-1}}\right)}\left(\sigma(\tau_a^V) - \tau_a^V > t,\>\tau_a^V < \tau_{-g/(1+\gamma)}^V\right)\right)\\
    \leq \underset{\nu \in (a, b)}{\sup}\> \prob_{(0, \nu)}\left(\sigma(\tau_a^V) - \tau_a^V > t, \> \tau_a^V < \tau_b^V \right)
    \leq e^{-ct},
\end{multline}
for all $t$ sufficiently large and some positive constant $c$. Using Lemma \ref{lemma:vel_increase_below_renewal_tailprob} with $\epsilon_0 = \frac{g}{4\gamma}, h=0$ and applying the strong Markov property at $\alpha_{3k + 1}$ gives for $k \geq 0$,
\begin{equation}\label{eqn:2renewal_above_time_diffs}
 \prob_{\mathsmaller{\renewalpt}}\left(\alpha_{3k + 2} - \alpha_{3k + 1} > t\right)  =   \prob_{\mathsmaller{\renewalpt}}\left(\alpha_{3k + 2} - \alpha_{3k + 1} > t, \,\, N^- \geq k+1 \right) 
\leq  e^C\, e^{-t\left(g/\gamma\right)^2/32 } .
\end{equation}
Apply Lemma \ref{lemma:vel_escape_interval} for the tail bound on the escape time of the interval $\left[a, -g/(1+\gamma)\right]$,
\begin{equation} \label{eqn:3renewal_above_time_diffs}
    \prob_{\mathsmaller{\renewalpt}}\left(\alpha_{3k+3} - \alpha_{3k +2} > t \right) \leq p^t,
\end{equation}
holds for some $p \in (0, 1)$ depending only on $\gamma, g$ and for all $t$ sufficiently large. Combining \eqref{eqn:0renewal_above_time_diffs}, \eqref{eqn:1renewal_above_time_diffs}, \eqref{eqn:2renewal_above_time_diffs} and \eqref{eqn:3renewal_above_time_diffs} completes the proof.
\end{proof}

\begin{lemma} \label{lemma:renewal_tail_from_below} There exist  positive constants $ t_0'(\gamma, g), c'$ such that
$$
\prob_{\mathsmaller{\renewalpt}}\left(\zeta > t^2,  \tau_a^V < \tau_b^V\right)  = \prob_{\mathsmaller{\renewalpt}}\left(\alpha_{3N^-} > t^2,  \tau_a^V < \tau_b^V\right) \leq e^{-c'\, t},$$
for all $t > t_0'(\gamma, g)$.
\end{lemma}
\begin{proof}
For any $k \ge 1$, if $N^- \leq k$, then $\alpha_{3(j+1)} = \alpha_{3j}$ for $j \geq k$. So for $t > 0, n \geq 1$, the union bound shows

\begin{equation} \label{eqn:final1_lemma:renewal_tail_from_below}
\prob_{\mathsmaller{\renewalpt}}\left(1 \leq N^- \leq n, \,\,\, \underset{1\leq k \leq (N^- - 1)}{\sup} (\alpha_{3(k+1)} - \alpha_{3k}) > t \right) \leq \sum_{k=0}^{n-1} \prob_{\mathsmaller{\renewalpt}} \left(\alpha_{3(k+1)} - \alpha_{3k} > t , \,\,\, N^- \geq k + 1\right),
\end{equation}
and
\begin{multline} \label{eqn:final2_lemma:renewal_tail_from_below}
    \prob_{\mathsmaller{\renewalpt}}\left(\alpha_{3N^-} > t^2,  \,\, \tau_a^V < \tau_b^V\right) \leq  \prob_{\mathsmaller{\renewalpt}}\left(N^- > t\right) + \prob_{\mathsmaller{\renewalpt}}\left(\sum_{k=0}^{N^- - 1}(\alpha_{3(k+1)} - \alpha_{3k}) > t^2, \,\, 0 < N^- \leq t \right) \\ \leq \prob_{\mathsmaller{\renewalpt}}\left(N^- > t\right) + \prob_{\mathsmaller{\renewalpt}}\left(1 \leq N^- \leq t, \,\,\, \underset{1\leq k \leq (N^- - 1)}{\sup} (\alpha_{3(k+1)} - \alpha_{3k}) > t \right).
\end{multline}
Using \eqref{eqn:final1_lemma:renewal_tail_from_below} in \eqref{eqn:final2_lemma:renewal_tail_from_below} and applying the bounds obtained in Lemma \ref{lemma:stop_renewal_from_below} and Lemma \ref{lemma:renewal_above_time_diffs}, we obtain
\begin{multline*}
 \prob_{\mathsmaller{\renewalpt}}\left(\alpha_{3N^-} > t^2,  \,\, \tau_a^V < \tau_b^V\right) \leq\prob_{\mathsmaller{\renewalpt}}\left(N^- > t\right) + \sum_{k=0}^{\lfloor t\rfloor-1} \prob_{\mathsmaller{\renewalpt}} \left(\alpha_{3(k+1)} - \alpha_{3k} > t , \,\,\, N^- \geq k + 1\right)\\
 \le  \left( 1 - \frac{1}{\sqrt{2\pi}}\frac{\frac{g}{\sqrt{2\gamma}(1+\gamma)}}{\left(\frac{g}{\sqrt{2\gamma}(1+\gamma)}\right)^2 + 1}\, \exp\left\{-\frac{g^2}{4\gamma(1+\gamma)^2}\right\}\right)^{t-1} + te^C \, e^{-c\, (t-1)},
\end{multline*}
which gives the bound stated in the lemma for sufficiently large $t$.
\end{proof}

\subsection{Renewal point approached from above}\label{sec:renewal_from_above} 
Our goal in this section is to bound $\prob_{\mathsmaller{\renewalpt}}\left(\zeta > t^2,  \tau_b^V < \tau_a^V\right) $, the case where $V$ rises to the level $b > \frac{-g}{1+\gamma}$ before the process returns to $\renewalpt$. Define $\beta_{-1}=0$ and $\beta_0 = \tau^V_a \wedge \tau^V_b$. If $\tau_a^V < \tau_b^V$, define $\beta_j = \beta_0$ for all $k \ge 0$ and $N^{+} = 0$. If $V_{\beta_{3k}} = b$ and $k \geq 0$, 
\begin{eqnarray}  \label{eqn:def_stop_times_above}
\beta_{3k + 1} &=& \inf\left\{t \geq \beta_{3k} \,|\,\, V_t = -\frac{g-\frac{g}{4(1+\gamma)}}{1+\gamma}\right\}  \nonumber \\
\beta_{3k + 2} &=& \sigma\left(\beta_{3k + 1}\right) \, \wedge \, \inf\left\{t \geq \beta_{3k + 1} \,|\,\, V_t = -\frac{g}{1+\gamma}\right\} \nonumber \\
\beta_{3k + 3} &=& \inf\left\{t \geq \beta_{3k + 2} \,|\,\, V_t = -\frac{g}{1+\gamma} \quad \text{or} \quad b\right\} \nonumber \\
N^+ &=& \inf\left\{k \geq 1 \, | \,\, V_{\beta_{3k}} = -\frac{g}{1+\gamma}\right\}.
\end{eqnarray}
If $V_{\beta_{3k}} = -\frac{g}{\gamma + 1}$, then $\beta_j = \beta_{3k} \quad j \geq 3k$. The fact that almost surely, $\beta_j < \infty$ for all $j$ and $N^+< \infty$ can be shown using arguments similar to those succeeding \eqref{eqn:def_stop_times_below}. 

For Section \ref{sec:renewal_from_below}, in which $\tau_a^V < \tau_b^V$, we relied on the property \eqref{eqn:increase_on_boundary_only} to imply that $H=0$ at the time when $V$ increases to $\frac{-g}{1+\gamma}$ from below. No such claim is possible when $\tau_b^V < \tau_a^V$ as $V$ crosses $\frac{-g}{1+\gamma}$ from above. In this case, $\beta_{3N^+} \leq \zeta$, with equality only on the event $H_{\beta_{3N^+}} = 0$.  

We therefore break this section into two parts: First, we derive bounds like those in Lemma \ref{lemma:renewal_above_time_diffs} but for the differences $\left\{\beta_j - \beta_{j-1} \right\}_{j\geq 0}$ until the time $\beta_{3N^+}$, when $V$ crosses $\frac{-g}{1+\gamma}$ from above. We then control the time taken after $\beta_{3N^+}$ for the velocity to return to $\frac{-g}{1+\gamma}$ from below.
\begin{lemma}\label{lemma:renewal_tail_from_above} There exists a positive constants $t_0(\gamma, g), c$ such that
$$\prob_{\mathsmaller{\renewalpt}}\left(\beta_{3N^+} > t^2,  \tau_b^V < \tau_a^V\right) \leq e^{-c\, t},$$
for all $t > t_0(\gamma, g)$.
\end{lemma}
\begin{proof}
Consider Lemma \ref{lemma:vel_drop_above_tailbound} with $u = -\frac{g - \frac{g}{4(1+\gamma)}}{1+\gamma}$ and $\nu = -\frac{g-\frac{g}{2(1+\gamma)}}{1+\gamma}$, so that $(1+\gamma)u + g = g/4(1+\gamma)$ and $u - \nu = -g/4(1+\gamma)^2$. The lemma shows there exist positive constants $c$ and $t_0(\gamma, g)$ such that
\begin{equation} \label{eqn:diff_betas_above}
    \prob_{\mathsmaller{\renewalpt}}\left(\beta_{3k + 1} - \beta_{3k} > t, N^+ \geq k+1 \right)\leq e^{-ct},
\end{equation}
for $t > t_0(\gamma, g)$. On $N^+ \geq k+1$, for $s \in [\beta_{3k + 1}, \beta_{3k + 2}]$, $L_s$ is constant and the velocity decreases deterministically from $-\frac{g-\frac{g}{4(1+\gamma)}}{1+\gamma}$ at $\beta_{3k + 1}$ to some point bounded below by $-\frac{g}{1+\gamma}$ at time $\beta_{3k + 2}$. This observation, along with \eqref{eqn:system}, implies
\begin{equation*}
    -\frac{g}{1+\gamma} + \frac{g-\frac{g}{4(1+\gamma)}}{1+\gamma}\leq V_{\beta_{3k + 2}} - V_{\beta_{3k + 1}} \leq \left(\beta_{3k + 2} - \beta_{3k+1}\right) \left(-\frac{g}{1+\gamma} \right).
\end{equation*}
This implies
\begin{equation}\label{downcr2}
\beta_{3k + 2} - \beta_{3k + 1}\le \frac{1}{4(1+\gamma)}.
\end{equation}
For any $k\ge 0$, the strong Markov property at $\beta_{3k+2}$ and Lemma \ref{lemma:vel_escape_interval} produce a constant $\overset{\sim}{p} \in (0, 1)$ depending only on $\gamma, g$ such that
\begin{equation}\label{downcr3}
  \prob_{\mathsmaller{\renewalpt}}\left(\beta_{3k+3} - \beta_{3k+2} > t \right) \leq \overset{\sim}{p}^t,
\end{equation}
when $t$ is sufficiently large. Arguments similar to those in Lemma \ref{lemma:stop_renewal_from_below} show for sufficiently large $n \ge 1$
\begin{equation}\label{Ncr}
    \prob_{\mathsmaller{\renewalpt}}\left(N^+ > n\right) \leq e^{-c n}
\end{equation}
for a positive constant $c$. Using \eqref{eqn:diff_betas_above}, \eqref{downcr2}, \eqref{downcr3} and \eqref{Ncr} in bounds along the lines of \eqref{eqn:final1_lemma:renewal_tail_from_below} and \eqref{eqn:final2_lemma:renewal_tail_from_below} with $\alpha_k$ replaced by $\beta_k$ for each $k \ge 0$ and $N^-$ replaced by $N^+$, the lemma follows.
\end{proof}
We now consider the time needed for $V$ to return to $-g/(1+\gamma)$ after $\beta_{3N^+}$. On $\tau_a^V < \tau_b^V$ define $\alphavar_j = \tau^V_a$ for $j \geq -1$ and $N^{\sharp}=-1$. On $\tau_b^V < \tau_a^V$ define 
\begin{eqnarray*}
\alphavar_{-1} &=& \inf\left\{t \geq \beta_{3N^+} \,|\,\, H_t= 0 \quad \text{or} \quad V_t  = a\right\}\\
\alphavar_0 &=& \inf\left\{t \geq \alphavar_{-1} \,|\,\, V_t = -\frac{g}{1 + \gamma} \quad \text{or} \quad V_t = a\right\} 
\end{eqnarray*}
If $V_{\alphavar_0} = -g/(\gamma+1)$, define $\alphavar_j = \alphavar_0$ for all $j \ge 1$, otherwise define $\{\alphavar_j\}_{j \geq 1}$ analogously to $\{\alpha_j\}_{j \geq 1}$ in \eqref{eqn:def_stop_times_below}, with $\alphavar_1 = \sigma(\alphavar_0)$ etc. Also define
\begin{equation*}
N^\sharp = \inf\left\{k \geq 0 \,|\,\, V_{\alphavar_{3k}} = -\frac{g}{\gamma + 1}\right\},
\end{equation*}
from which it follows that $\zeta = \alphavar_{3N^\sharp}$ when $\tau_b^V < \tau_a^V$.
\begin{lemma} \label{lemma:renewal_tail_2nd_below} There exist constants $t_0(\gamma, g), c > 0$ such that
$$\prob_{\mathsmaller{\renewalpt}}\left(\zeta - \beta_{3N^+} > t^2,  \tau_b^V < \tau_a^V\right) = \prob_{\mathsmaller{\renewalpt}}\left(\alphavar_{3N^\sharp} - \beta_{3N^+} > t^2,  \tau_b^V < \tau_a^V\right) \leq e^{-c\, t},$$
for $t > t_0(\gamma, g)$.
\end{lemma}
\begin{proof}
First note that if $ \tau_b^V < \tau_a^V$ and $H_{\beta_{3N^+}} = 0$, then $\beta_{3N^+} = \alphavar_{-1} = \alphavar_0$ and $N^\sharp = 0$. Since $V_{\beta_{3N^+}} = -g/(1+\gamma)$, in this case we have $\alphavar_{3N^\sharp} = \beta_{3N^\sharp} = \beta_{3N^+}= \zeta$. 

Henceforth we consider the case $\tau_b^V < \tau_a^V, H_{\beta_{3N^+}} > 0$. Arguing as in Lemma \ref{lemma:stop_renewal_from_below}, we have for $n \geq 0$ and a positive constant $c$,
\begin{equation}\label{eqn:renewal_2nd_0}
    \prob_{\mathsmaller{\renewalpt}}\left(N^\sharp > n \right) \leq e^{-cn}.
\end{equation}
Since $L_s$ is constant on $[\beta_{3N^+}, \alphavar_{-1}]$, we proceed as in \eqref{downcr2} to show there exists a constant $C > 0$ such that
\begin{equation}\label{eqn:renewal_2nd_1}
    \alphavar_{-1} - \beta_{3N^+} \leq C.
\end{equation}
In the following analysis, there are two cases to consider: $H_{\alphavar_{-1}} = 0$ and $H_{\alphavar_{-1}} > 0$. In the former situation, $V_{\alphavar_{-1}} \in (a, -g/(1+\gamma))$, so we use Lemma \ref{lemma:vel_escape_interval} and the strong Markov property at $\alphavar_{-1}$ to show there exists $\bar{p} \in (0, 1)$ depending on $\gamma, g$ such that for $t$ sufficiently large,
\begin{equation}\label{eqn:renewal_2nd_2}
    \prob_{\mathsmaller{\renewalpt}}\left(\alphavar_0 - \alphavar_{-1} > t, \>H_{\alphavar_{-1}} = 0 \right) \leq \underset{\nu \in (a, -g/(1+\gamma))}{\sup}\>\prob_{(0, \nu)}\left(t < \tau_a^V \wedge \tau_{-g/(1+\gamma)}^V \right) \leq \bar{p}^t.
\end{equation}
The analysis of $\{\alphavar_{j+1} - \alphavar_{j}\}_{j \ge 0}$ can be done exactly as that of $\{\alpha_{j+1} - \alpha_j\}_{j \ge 3}$ performed in Lemma \ref{lemma:renewal_above_time_diffs} using Lemmas \ref{lemma:vel_escape_interval}, \ref{lemma:gap_hits_zero_before_long} and \ref{lemma:vel_increase_below_renewal_tailprob}. Thus we proceed as in Lemma \ref{lemma:renewal_tail_from_below}, using $N^\sharp$ in place of $N^-$ and \eqref{eqn:renewal_2nd_0} instead of Lemma \ref{lemma:stop_renewal_from_below}, to obtain a positive constant $c$ such that for $t$ sufficiently large
\begin{equation}\label{eqn:renewal_2nd_3}
    \prob_{\mathsmaller{\renewalpt}}\left(\alphavar_{3N^\sharp} - \beta_{3N^+} > t^2,  \tau_b^V < \tau_a^V, \>H_{\alphavar_{-1}} = 0 \right) \leq e^{-ct}.
\end{equation}
Now consider $H_{\alphavar_{-1}} > 0$. In that case, $\alphavar_{-1} = \alphavar_0 = \inf\{s \geq \beta_{3N^+}\> | \> V_s = a \}$. The methods of Lemma \ref{lemma:renewal_above_time_diffs} fail to bound the probability of $\alpha_1 - \alpha_0 > t$ since in principle $H_{\alphavar_{-1}}$ might be quite large. We first apply the union bound to show for $k \geq 1$,
\begin{multline}\label{eqn:renewal_2nd_4}
    \prob_{\mathsmaller{\renewalpt}}\left(\alphavar_1 - \alphavar_0 > t, \> H_{\alphavar_{-1}} > 0, \> N^+ = k\right)
    \leq \prob_{\mathsmaller{\renewalpt}}\left(N^+ \geq k,\> \beta_{3k-2} - \beta_{3k-3} > t \right) \\
    + \prob_{\mathsmaller{\renewalpt}}\left(\alphavar_1 - \alphavar_0 > t, \> H_{\alphavar_{-1}} > 0, \> N^+ = k,\>\beta_{3k-2} - \beta_{3k-3} \leq t \right) \leq  \prob_{\mathsmaller{\renewalpt}}\left(N^+ \geq k,\> \beta_{3k-2} - \beta_{3k-3} > t \right) \\
    + \prob_{\mathsmaller{\renewalpt}}\left(\alphavar_1 - \alphavar_0 > t, \> H_{\alphavar_{-1}} > 0, \> N^+ = k, \> \beta_{3k-2} - \beta_{3k-3} \leq t, \> H_{\beta_{3k-2}} \le tg/4\gamma(1+\gamma) \right) \\
    + \prob_{\mathsmaller{\renewalpt}}\left(\alphavar_1 - \alphavar_0 > t, \> H_{\alphavar_{-1}} > 0, \> N^+ = k, \> \beta_{3k-2} - \beta_{3k-3} \leq t, \> H_{\beta_{3k-2}} > tg/4\gamma(1+\gamma) \right).
\end{multline}
Recall that $\beta_{3k - 3}$ is a point of increase of $V$ to the level $b= -\left(g-g/2(1+\gamma) \right)/ (1+\gamma)$, so \eqref{eqn:increase_on_boundary_only} shows $H_{\beta_{3k - 3}} = 0$. $\beta_{3k-2}$ is the first time after $\beta_{3k - 3}$ the velocity falls to $-\left(g-g/4(1+\gamma) \right)/ (1+\gamma)$. Using Lemma \ref{lemma:vel_drop_above_tailbound} and the strong Markov property at $\beta_{3k-3}$, there exists constants $c, t_0(\gamma, g) > 0$ such that for all $t > t_0(\gamma, g)$,
\begin{multline}\label{eqn:renewal_2nd_5}
    \prob_{\mathsmaller{\renewalpt}}\left(N^+ \geq k,\> \beta_{3k-2} - \beta_{3k-3} > t \right) = \expec_{\mathsmaller{\renewalpt}}\left(N^+ \geq k,\> \prob_{(H_{\beta_{3k}}, V_{\beta_{3k}})}\left(\tau_{-\left(g-g/4(1+\gamma) \right)/ (1+\gamma)}^V > t \right) \right)\\
    \leq \prob_{\mathsmaller{\renewalpt}}\left(N^+ \geq k\right) \> e^{-ct}.
\end{multline}
When $N^+ = k$ and $H_{\beta_{3N^+}}, H_{\alphavar_{-1}} > 0$, the velocity starting from $\beta_{3k-3}$ makes an excursion from $b$ to $-\frac{g - g/4(1+\gamma)}{1+\gamma} $ at time $\beta_{3k-2}$, then to $a$ without returning to $b$. Therefore, 
\begin{equation}\label{eqn:renewal_2nd_6}
    \alphavar_0 = \alphavar_{-1} = \inf\{s \geq \beta_{3k - 2}\> | \> V_s = a \} < \inf\{s \geq \beta_{3k - 2}\> | \> V_s = b \}, \quad \quad \text{when}\> \  N^+ = k.
\end{equation}
Using \eqref{eqn:renewal_2nd_6}, the strong Markov property at $\beta_{3k-2}$ and Lemma \ref{lemma:gap_hits_zero_before_long} we obtain $t_1(\gamma, g) > 0$, which we take to be larger than $t_0(\gamma, g)$, such that for $t > t_1(\gamma, g)$,
\begin{multline}\label{eqn:renewal_2nd_7}
    \prob_{\mathsmaller{\renewalpt}}\left(\alphavar_1 - \alphavar_0 > t, \> H_{\alphavar_{-1}} > 0, \> N^+ = k, \beta_{3k-2} - \beta_{3k-3} \leq t, \> H_{\beta_{3k-2}} \le tg/4\gamma(1+\gamma) \right) \\
    \leq  \prob_{\mathsmaller{\renewalpt}}\left(N^+ \geq k\right)\> \underset{(h, \nu) \in [0, tg/4\gamma(1+\gamma)]\times(a, b)}{\sup}\> \prob_{(h, \nu)}\left(\sigma(\tau_a^V) - \tau_a^V > t, \> \tau_a^V < \tau_b^V \right) \leq \prob_{\mathsmaller{\renewalpt}}\left(N^+ \geq k\right) \> e^{-ct}.
\end{multline}
Fix $t > t_1(\gamma, g)$ and set $(H_0, V_0) = (0, b)$. System equations \eqref{eqn:system} show $H_u + V_u/\gamma = b/\gamma - ug/\gamma - B_u +(1+1/\gamma)L_u$. When $u < t \wedge \tau_{-\left(g-g/4(1+\gamma) \right)/ (1+\gamma)}^V$, we have $S_u \geq -u\left(g-g/4(1+\gamma) \right)/ (1+\gamma)$ and $L_u \leq \sup_{s \leq u}\left(B_s + s\left(g-g/4(1+\gamma) \right)/ (1+\gamma)\right) \leq \sup_{s \leq u}\>B_s + u\left(g-g/4(1+\gamma) \right)/ (1+\gamma) $. As a result, with $c' = b/\gamma + \left(g-g/4(1+\gamma) \right)/\gamma (1+\gamma)> 0$,
\begin{multline}\label{eqn:renewal_2nd_8}
    H_u = H_u + V_u / \gamma - V_u/\gamma \leq c' - B_u - ug/\gamma + (1+1/\gamma)L_u\\
    \leq c' - B_u + u\left((1+1/\gamma)\left(g-g/4(1+\gamma) \right)/ (1+\gamma) - g/\gamma\right) + (1+1/\gamma)\sup_{s \leq u}\>B_s \\
    = c' - B_u - ug/4\gamma(1+\gamma) + (1+1/\gamma)\sup_{s \leq u}\>B_s \leq c' + \sup_{s \leq t}(-B_s) + (1+1/\gamma)\sup_{s \leq t}\>B_s.
\end{multline}
From \eqref{eqn:renewal_2nd_8}, we conclude that if $\tau_{-\left(g-g/4(1+\gamma) \right)/ (1+\gamma)}^V \leq t$, then 
$
H_{\tau_{-\left(g-g/4(1+\gamma) \right)/ (1+\gamma)}^V} > tg/4\gamma(1+\gamma)
$ 
implies $\sup_{s \leq t}(-B_s) + (1+1/\gamma)\sup_{s \leq t}\>B_s> tg/4\gamma(1+\gamma) - c' $. We choose $t_1(\gamma, g)$ large enough that $t_1(\gamma, g) - c' > 0$.  Now the strong Markov property at $\beta_{3k-3}$, \eqref{eqn:renewal_2nd_8} and Gaussian tail bounds show there exists a $t_2(\gamma, g) > t_1(\gamma, g)$ and constants $C, C', c, c' > 0$ such that that for $t > t_2(\gamma, g)$,
\begin{multline}\label{eqn:renewal_2nd_9}
    \prob_{\mathsmaller{\renewalpt}}\left(\alphavar_1 - \alphavar_0 > t, \> H_{\alphavar_{-1}} > 0, \> N^+ = k, \beta_{3k-2} - \beta_{3k-3} \leq t, \> H_{\beta_{3k-2}} > tg/4\gamma(1+\gamma) \right) \\
    \leq \prob_{\mathsmaller{\renewalpt}}\left(N^+ \geq k\right) \> \prob_{(0, b)}\left(H_{\tau_{-\frac{g-g/4(1+\gamma)}{1+\gamma}}^V} > tg/4\gamma(1+\gamma), \> \tau_{-\left(g-g/4(1+\gamma) \right)/ (1+\gamma)}^V \leq t \right)\\
    \leq \prob_{\mathsmaller{\renewalpt}}\left(N^+ \geq k\right) \> \left[\prob\left(\sup_{s \leq t}\left(-B_s \right) > tg/8\gamma(1+\gamma) - c'/2 \right)\right. \qquad\qquad\qquad\qquad\qquad\\
    \qquad\qquad\left. + \prob\left(\sup_{s \leq t}\left(B_s \right) > \frac{\gamma}{1+\gamma}\left(tg/8\gamma(1+\gamma) - c'/2 \right)\right)\right] \\
    \leq \prob_{\mathsmaller{\renewalpt}}\left(N^+ \geq k\right)\>\left[e^{C'}\>e^{-c't} +  e^C\>e^{-ct}\right].
\end{multline}
By \eqref{Ncr}, $\expec_{\mathsmaller{\renewalpt}} \> N^+ \leq e^C$ for a positive constant $C$. We apply \eqref{eqn:renewal_2nd_5}, \eqref{eqn:renewal_2nd_7} and \eqref{eqn:renewal_2nd_9} to \eqref{eqn:renewal_2nd_4} and sum over $k$ to obtain positive constants $c$, $t_2(\gamma, g)$ such that,
\begin{multline}\label{eqn:renewal_2nd_10}
    \prob_{\mathsmaller{\renewalpt}}\left(\alphavar_1 - \alphavar_0 > t, \> H_{\alphavar_{-1}} > 0\right) = \sum_{k = 1}^\infty \prob_{\mathsmaller{\renewalpt}}\left(\alphavar_1 - \alphavar_0 > t, \> H_{\alphavar_{-1}} > 0, \> N^+ = k, \> N^+ \geq k\right) \\
    \leq e^{-ct} \> _{k = 1}^\infty \prob_{\mathsmaller{\renewalpt}}\left(N^+  \ge k\right) \leq e^C \> e^{-ct},
\end{multline}
for $t > t_2(\gamma, g)$. Now the analysis of $\{\alphavar_{j+1} - \alphavar_{j}\}_{j \ge 1}$ can be done exactly as that of $\{\alpha_{j+1} - \alpha_j\}_{j \ge 1}$ performed in Lemma \ref{lemma:renewal_above_time_diffs}. Once again we argue as in Lemma \ref{lemma:renewal_tail_from_below}, using $N^\sharp$ in place of $N^-$ and \eqref{eqn:renewal_2nd_0} instead of Lemma \ref{lemma:stop_renewal_from_below}, to obtain a positive constant $c$ such that for $t$ sufficiently large
\begin{equation}\label{eqn:renewal_2nd_11}
    \prob_{\mathsmaller{\renewalpt}}\left(\alphavar_{3N^\sharp} - \beta_{3N^+} > t^2,  \tau_b^V < \tau_a^V, \>H_{\alphavar_{-1}} > 0 \right) \leq e^{-ct}.
\end{equation}
\eqref{eqn:renewal_2nd_3} and \eqref{eqn:renewal_2nd_11} prove the lemma.
\end{proof}

Now, we have all the tools needed to prove Theorems \ref{thm:renewal} and \ref{thm:statmeas_renewal}.
\begin{proof}[Proof of Theorem \ref{thm:renewal}]
This is a direct consequence of Lemmas \ref{lemma:renewal_tail_from_below}, \ref{lemma:renewal_tail_from_above} and \ref{lemma:renewal_tail_2nd_below}.
\end{proof}

\begin{proof}[Proof of Theorem \ref{thm:statmeas_renewal}] Theorem \ref{thm:renewal} and the strong Markov property, proven to hold in Theorem \ref{thm:existence}, show that the system is `classical regenerative' and satisfies the conditions of Theorem 2.1, Chapter 10, in \cite{thor}, which gives a stationary measure of the stated form. 

Now we prove the second claim of the theorem, which also implies uniqueness of the stationary measure. Define $N_t = \sup\{k \geq 0 \> : \> \zeta_k \leq t \}$, the number of renewals before time $t$. It is enough to show the claim for bounded, non-negative $f$. Recalling $\zeta = \zeta_0$,
\begin{multline}\label{ergineq}
    \int_0^{\zeta\wedge t} f(H_s, V_s)\> ds + \indi{\zeta \leq t} \sum_{k=1}^{N_t}\> \int_{\zeta_{k-1}}^{\zeta_k} f(H_s, V_s)\> ds \leq \int_0^t f(H_s, V_s)\> ds \\
    \leq \int_0^{\zeta} f(H_s, V_s)\> ds + \sum_{k=1}^{N_t + 1}\> \int_{\zeta_{k-1}}^{\zeta_k} f(H_s, V_s)\> ds.
\end{multline}
Now the Strong Law of Large Numbers implies $N_t/t \to 1/\expec_{\mathsmaller{\renewalpt}}( \zeta ) > 0$ (e.g. Theorem 2.4.6 of \cite{durrett}). Using \eqref{eqn:renewal_sequence} and applying the Law of Large Numbers in \eqref{ergineq} completes the proof.
\end{proof}
\section{Fluctuation bounds on a renewal interval}\label{flucinter}
We show that on the interval $[0, \zeta]$, the probability of the velocity hitting a large value $y$ has Gaussian decay in $y$ and the corresponding probability of the gap hitting a large value $x$ has exponential decay in $x$. These estimates directly imply bounds on the tails of the stationary measure given in Theorem \ref{thm:tails_under_stationarity} via the representation in Theorem \ref{thm:statmeas_renewal}, and also produce the oscillation estimates stated in Theorem \ref{thm:fluctuations}. 

\subsection{Velocity bounds} \label{sec:vel_bounds}
\begin{lemma} \label{lemma:vel_hits_large_before_small}
There exists a positive constant $ c $ such that $$ \prob_{(0, y)}\left(\tau_{2y}^V < \tau^V_{y/2}\right) \leq e^{c}e^{-(1+\gamma)(y + \frac{g}{1+\gamma})^2}, $$
for all $y > 0$.
\end{lemma}

\begin{proof} Define $F_t = y + \sup_{u\leq t}\left(B_u - uy/2\right) - (g + y\gamma / 2)t$. For $t < \tau^V_{y/2}\wedge\tau_{2y}^V$, the system equations \eqref{eqn:system} and the Skorohod representation \eqref{sm} imply $V_t \leq F_t $. Hence,
\begin{multline*}
    \prob_{(0, y)}\left(\tau_{2y}^V < \tau^V_{y/2}\right) \leq  \prob\left(F_t \> \text{hits 2y before y/2}\right) 
\leq \prob\left(\underset{t < \infty}{\sup}(B_t - ty/2 - (g + y\gamma / 2)t ) \geq y\right) \\
= \exp\left\{-2y(g + (1+\gamma)y/2) \right\} = e^{\frac{g^2}{1 + \gamma}}e^{-(1+\gamma)(y + \frac{g}{1+\gamma})^2} .
\end{multline*}
The second inequality comes from the fact that the time at which $F_t$ hits $2y$ must be a point of increase of $\sup_{u\leq t}\left(B_u - uy/2\right)$. The first equality uses the fact that for any $u > 0$, $\underset{s < \infty}{\sup}\left(B_s - su \right) \overset{d}{=} \text{Exponential}(2u)$ (see Chapter 3.5 of \cite{kar}).
\end{proof}

Fix any $y > 0$ and choose the starting configuration $(H_0,V_0) = (0,y)$. Define the following sequence of stopping times: $\tau_0 = 0$ and for $k \geq 0$,
\begin{eqnarray} \label{eqn:def_bubble_stop_times}
\tau_{2k + 1} &=& \inf\{t \geq \tau_{2k} \, | \,\, V_t = 2y \quad \text{or} \quad y/2\} \quad \quad \text{if } V_{\tau_{2k}} \neq -g/(1+\gamma), 2y\nonumber \\
&=& \tau_{2k} \quad \quad \quad \text{otherwise} \nonumber \\
\tau_{2k + 2} &=& \inf\{t \geq \tau_{2k} \, | \,\, V_t = y \quad \text{or} \quad -g/(1+\gamma)\} \quad \quad \text{if } V_{\tau_{2k + 1}} \neq -g/(1+\gamma), 2y\nonumber \\
&=& \tau_{2k + 1} \quad \quad \text{otherwise} \nonumber \\
J_y &=& \min\{k \geq 1 \, | \,\, V_{\tau_{2k}} = -g/(1+\gamma) \quad  \text{  or  } \quad 2y \} .
\end{eqnarray}

\begin{lemma} \label{lemma:bubble_excursion_number} There exist positive constants $y'(\gamma, g)$ and $p(\gamma,g) \in (0,1)$ such that $$\prob_{(0,y)}\left(J_y > n\right) \leq p(\gamma,g)^n, $$ for $y > y'(\gamma, g)$ and $n \geq 0$.
\end{lemma}

\begin{proof}
For any $y>g/(1+\gamma)$, applying the strong Markov property at $\tau_1$, observe that 
\begin{multline}\label{bubble1}
    \prob_{(0,y)}\left(J_y > 1\right) = \prob_{(0,y)}\left(V_{\tau_1}=y/2, \ V_{\tau_2} = y \right)
    \leq \prob_{(0, y/2)}\left(\tau_y^V < \tau_{-\frac{g}{1+\gamma}}^V\right) \\ 
    \leq \prob_{(0, y/2)}\left(\tau_y^V < \tau_{\frac{g}{1+\gamma}}^V\right) + \prob_{(0, y)}\left(\tau^V_{\frac{g}{1+\gamma}} < \tau_y^V < \tau_{-\frac{g}{1+\gamma}}^V\right).
\end{multline}
We bound the second probability on the right hand side by applying the strong Markov property at the stopping time $\inf\left\lbrace t \ge \sigma\left(\tau^V_{\frac{g}{1+\gamma}}\right) \, \Big| \, V_t = g/(1+\gamma)\right\rbrace$ to obtain
$$
\prob_{(0, y)}\left(\tau^V_{\frac{g}{1+\gamma}} < \tau_y^V < \tau_{-\frac{g}{1+\gamma}}^V\right) \le \prob_{(0, \frac{g}{1+\gamma})}\left(\tau_y^V < \tau_{-\frac{g}{1+\gamma}}^V\right).
$$
Using this in \eqref{bubble1}, we obtain
\begin{equation}\label{bubble2}
\prob_{(0,y)}\left(J_y > 1\right)  \le \prob_{(0, y/2)}\left(\tau_y^V < \tau_{\frac{g}{1+\gamma}}^V\right) + \prob_{(0, \frac{g}{1+\gamma})}\left(\tau_y^V < \tau_{-\frac{g}{1+\gamma}}^V\right).
\end{equation}
To estimate the first probability on the right hand side of \eqref{bubble2}, note that if $V_0 = y/2$ and $t < \tau_y^V \wedge \tau_{\frac{g}{1+\gamma}}^V$, $$V_t \leq y/2 - t\left( \gamma \frac{g}{1+\gamma} + g\right) + \underset{s\leq t}{\sup}\left(B_s - s\frac{g}{1+\gamma} \right) \coloneqq y/2 + Z_t.$$
Arguing as in Lemma \ref{lemma:vel_hits_large_before_small},
\begin{multline}\label{bubble3}
    \prob_{(0, y/2)}\left(\tau_y^V < \tau_{\frac{g}{1+\gamma}}^V\right) 
    \leq \prob\left(y/2 + Z_t \quad \text{hits  } y \text{  before  } \frac{g}{1+\gamma} \right)\\
    \leq \prob\left(\underset{t < \infty}{\sup}\left(B_t - 2gt \right) > y/2 \right) = e^{-2gy}.
\end{multline}
To estimate the second probability in \eqref{bubble2}, observe that for any $y \ge 2g/(1+\gamma)$,
\begin{multline}\label{bubble4}
\prob_{(0, \frac{g}{1+\gamma})}\left(\tau_{-\frac{g}{1+\gamma}}^V < \tau_y^V\right) \ge \prob_{(0, \frac{g}{1+\gamma})}\left(\tau_{-\frac{g}{1+\gamma}}^V < \tau_{\frac{2g}{1+\gamma}}^V\right) \ge \prob_{(0, \frac{g}{1+\gamma})}\left(\tau^H_1 < \tau_{\frac{2g}{1+\gamma}}^V,  \ \tau^V_{(-g/\gamma, -g/(1+\gamma)]} \in \left[\tau^H_1, \tau_{\frac{2g}{1+\gamma}}^V\right] \right) \\
\ge \prob_{(0, \frac{g}{1+\gamma})}\left(\tau^H_1 < \tau_{\frac{2g}{1+\gamma}}^V\right) \times \inf_{v \in ( -g/(1+\gamma), 2g/(1+\gamma))}\mathbb{P}_{(1,v)}\left(\tau_{-\frac{g}{1+\gamma}}^V < \tau_{\frac{2g}{1+\gamma}}^V\right)
\end{multline}
where the last inequality follows by applying the strong Markov property at $\tau^H_1$. Recalling that for any $t \ge 0$
\begin{equation}\label{Hlow}
H_t\ge H_0 -\frac{gt}{\gamma}  - B_t + L_t \ge H_0 -\frac{gt}{\gamma}  - B_t,
\end{equation}
and for $V_0 = g/(1+\gamma)$, as $V_t >-g/\gamma$ for all $t$,
$$
V_t =V_0 - \int_0^t\left(\gamma V_s + g\right)ds + L_t \le \frac{g}{1+\gamma} + \sup_{u \le t}\left(-H_0 + B_u + \frac{gu}{\gamma}\right),
$$
we obtain
\begin{equation}\label{bubble5}
\prob_{(0, \frac{g}{1+\gamma})}\left(\tau^H_1 < \tau_{\frac{2g}{1+\gamma}}^V\right) \ge \mathbb{P}\left(B_1 \le -1 - g/\gamma, \ \sup_{u \le 1}\left(B_u + \frac{gu}{\gamma}\right) < \frac{g}{1+\gamma}\right) = p_1(\gamma,g)>0.
\end{equation}
Now we bound the second term in the product on the right hand side of \eqref{bubble4} from below. Choosing starting point $(H_0,V_0) = (1,v)$ for any $v \in ( -g/(1+\gamma), 2g/(1+\gamma))$, recall from \eqref{eqn:system_starts_interior} that for any $t>0$, if $\sigma(0)>t$, then
$$
V_t \le \left(\frac{2g}{1+\gamma} + g/\gamma\right)e^{-\gamma t} - g/\gamma.
$$
The right side equals $-g/(1+\gamma)$ when $t= \gamma^{-1}\log\left(1+3\gamma\right)$ and hence, if $\sigma(0) \ge \gamma^{-1}\log\left(1+3\gamma\right)$, then $\tau_{-\frac{g}{1+\gamma}}^V \le \gamma^{-1}\log\left(1+3\gamma\right)$. Thus, from \eqref{Hlow}, if $B_t \le \frac{1}{2} - \frac{gt}{\gamma}$ for all $t \le \gamma^{-1}\log\left(1+3\gamma\right)$, then $\sigma(0) \ge \gamma^{-1}\log\left(1+3\gamma\right)$ and thus $\sigma(0) \ge \tau_{-\frac{g}{1+\gamma}}^V$. As the velocity is strictly decreasing on the interval $[0,\sigma(0)]$, for any $v \in ( -g/(1+\gamma), 2g/(1+\gamma))$,
\begin{equation}\label{bubble6}
\mathbb{P}_{(1,v)}\left(\tau_{-\frac{g}{1+\gamma}}^V < \tau_{\frac{2g}{1+\gamma}}^V\right) \ge \mathbb{P}_{(1,v)}\left(\tau_{-\frac{g}{1+\gamma}}^V  \le \sigma(0) \right) \ge \mathbb{P}\left(\sup_{u \le\gamma^{-1}\log\left(1+3\gamma\right)}\left(B_u + \frac{gu}{\gamma}\right) \le \frac{1}{2}\right) = p_2(\gamma,g)>0.
\end{equation}
Using \eqref{bubble5} and \eqref{bubble6} in \eqref{bubble4}, we obtain
\begin{equation}\label{bubble7}
\prob_{(0, \frac{g}{1+\gamma})}\left(\tau_{-\frac{g}{1+\gamma}}^V < \tau_y^V\right) \ge p_1(\gamma,g)p_2(\gamma,g)>0
\end{equation}
Using \eqref{bubble3} and \eqref{bubble7} in \eqref{bubble2}
\begin{equation}\label{bubble8}
 \prob_{(0,y)}\left(J_y > 1\right) \le e^{-2gy} + (1- p_1(\gamma,g)p_2(\gamma,g)).
\end{equation}
Choosing any $p(\gamma,g) \in ((1- p_1(\gamma,g)p_2(\gamma,g)), 1)$ and $y'(\gamma,g) \ge 2g/(1+\gamma)$ such that the right hand side of \eqref{bubble8} is bounded above by $p(\gamma,g)$ for all $y \ge y'(\gamma,g)$, the assertion of the lemma follows for $n=1$.
The result for $n \ge 2$ follows by induction upon using the strong Markov property at $\tau_{2n-1}$.
\end{proof}

\begin{theorem} \label{thm:velocity_bounds} There exist positive constants $y'(\gamma, g), c, C$ such that
$$e^{-C} \> e^{-2(1+\gamma)\left(y+g/(1+\gamma)\right)^2} \leq \prob_{\mathsmaller{\renewalpt}}\left(\underset{[0, \zeta]}{\sup} \,\, V_t \geq y \right) \leq  e^{c}e^{-\frac{1+\gamma}{4}(y + g/(1+\gamma))^2} ,$$
for all $y > y'(\gamma, g)$.
\end{theorem}

\begin{proof} Consider $\{\tau_i \}_{i \geq 1}$ and $J_y$ defined in \eqref{eqn:def_bubble_stop_times}. Then, choosing $y'(\gamma,g)$ to be the same constant as in Lemma \ref{lemma:bubble_excursion_number}, for any $y > y'(\gamma,g)$, we apply the strong Markov property at stopping time $\tau^V_y$ to obtain
\begin{multline} \label{eqn:vel_upper_bd}
    \prob_{\mathsmaller{\renewalpt}}\left(\underset{[0, \zeta]}{\sup} \,\, V_t \geq 2y \right) \le \mathbb{P}_{(0,y)}\left(\tau^V_{2y} < \tau^V_{-g/(1+\gamma)}\right)
    \leq \sum_{k=0}^\infty \prob_{(0,y)}\left(\underset{[\tau_{2k}, \tau_{2k+1}]}{\sup} \,\, V_t \geq 2y, \,\, J_y > k \right).   
\end{multline}
By applying the strong Markov property at time $\tau_{2k}$,
$$
\prob_{(0,y)}\left(\underset{[\tau_{2k}, \tau_{2k+1}]}{\sup} \,\, V_t \geq 2y, \,\, J_y > k \right) = \prob_{(0, y)}\left(\tau_{2y}^V < \tau^V_{y/2}\right)\prob_{(0,y)}\left(J_y > k \right).
$$
Using this in \eqref{eqn:vel_upper_bd}, we obtain
$$
\prob_{\mathsmaller{\renewalpt}}\left(\underset{[0, \zeta]}{\sup} \,\, V_t \geq 2y \right) \le \prob_{(0, y)}\left(\tau_{2y}^V < \tau^V_{y/2}\right) \,\, \expec_{\mathsmaller{\renewalpt}} \> \left(J_y\right)
$$
from which the upper bound claimed in the theorem (with $2y$ in place of $y$) follows for $y > y'(\gamma, g)$ from Lemmas \ref{lemma:vel_hits_large_before_small} and \ref{lemma:bubble_excursion_number}.

For the lower bound, we first show for $b = -\frac{g - \frac{g}{2(1+\gamma)}}{1+ \gamma}$, as in the definition of $\zeta$, that for $y$ sufficiently large,
\begin{equation} \label{eqn:vel_lower_bd_interim0}
    \prob_{(0, b)}\left(\tau_y^V < \tau^V_{-\frac{g}{1+\gamma}} \right) \geq  e^{-2(1+\gamma)\left(y+\frac{g}{1+\gamma}\right)^2}.
\end{equation}
Choosing the starting point $(H_0,V_0) = (0,b)$, for $t < \tau^V_{-\frac{g}{1+\gamma}} \wedge \tau^V_y$, we have 
$$
V_t = b  - \int_0^t\left(\gamma V_s + g\right)ds + L_t \ge b - (\gamma y + g)t + \sup_{u \le t}\left(B_u - yu\right) \geq b - t\left(\left(1 + \gamma \right)y + g\right) + B_t \coloneqq F_t .
$$ 
We also note that if $s(v) = e^{2v\left(\left(1 + \gamma \right)y + g\right)}$, then $s(F_t)$ is a bounded martingale on $t < \tau^V_{-\frac{g}{1+\gamma}} \wedge \tau^V_y$. By the optional stopping theorem,
\begin{equation}\label{eqn:vel_lower_bd_interim}
    \prob_{(0, b)}\left(\tau_y^V < \tau^V_{-\frac{g}{1+\gamma}} \right) \geq \prob_{(0, b)}\left(F_t\text{ hits } y \text{ before  } -\frac{g}{1+\gamma} \right)
    = \frac{s(b) - s\left(-\frac{g}{1+\gamma}\right)}{s(y) - s\left(-\frac{g}{1+\gamma}\right)} \geq  e^{-2(1+\gamma)\left(y+\frac{g}{1+\gamma}\right)^2}
\end{equation}
for $y$ sufficiently large. This proves \eqref{eqn:vel_lower_bd_interim0}. Recall $a = - \frac{g + \frac{g}{2\gamma}}{1+\gamma} < -\frac{g}{1+\gamma}$ in the definition of $\zeta$. Note that if $\tau_a^V < \tau_b^V$, then the renewal time $\zeta$ corresponds to the first hitting time of the level $-g/(1+\gamma)$ by the velocity after time $\tau_a^V$ and hence, $\tau^V_y > \zeta$. Using this observation and the strong Markov property at $\tau^V_b$, we obtain
\begin{equation} \label{eqn:vel_lower_bd_interim2}
    \prob_{\mathsmaller{\renewalpt}}\left(\underset{[0, \zeta]}{\sup} \,\, V_t \geq y \right) = \prob_{\mathsmaller{\renewalpt}}\left(\tau^V_b < \tau^V_a \right) \> \prob_{(0, b)}\left(\tau_y^V < \tau^V_{-\frac{g}{1+\gamma}} \right).
\end{equation}
The lower bound in the theorem follows from \eqref{eqn:vel_lower_bd_interim0} and \eqref{eqn:vel_lower_bd_interim2}, with $C = -\log \prob_{\mathsmaller{\renewalpt}}\left(\tau^V_b < \tau^V_a \right)$. $C > 0$ is finite via arguments similar to those of \eqref{eqn:vel_lower_bd_interim}.
\end{proof}


\subsection{Gap bounds}

To bound tail probabilities of $H$, we split the path of the process by hitting times of the velocity. Doing so allows the gap to be controlled by Brownian motion with drift and its running maximum.

\begin{lemma} \label{lemma:gap_increases_before_zero} For any $x>0$,
$$
 \underset{\nu \in \left(-g/\gamma, \frac{\gamma x}{4} - g/\gamma\right] }{\sup} \> \prob_{(x/2, \nu)}\left(\tau_x^H < \sigma(0) \right) \leq   \, \exp\left\{- x g/2\gamma  \right\}
$$
and
$$
 \underset{h \geq (\gamma / 2g)\log(2), \> \nu > -g/\gamma }{\inf} \> \prob_{(h, \nu)}\left(\tau_x^H < \sigma(0) \right) \ge  \exp\left\{- 2x \> g/\gamma \right\}.
$$
\end{lemma}
\begin{proof} Take any $x>0$. On $t < \tau_x^H \wedge \sigma(0)$, with $(H_0,V_0) = (h,\nu)$, $H_t$ is dominated by $h + \frac{\nu}{\gamma} + \frac{g}{\gamma^2} - B_t - t g/\gamma$, using \eqref{eqn:system_starts_interior}. For the upper bound, use $h = x/2$. The function $s(x) = e^{2\frac{g}{\gamma}x}$ makes $s(z - B_t - \frac{g}{\gamma}t )$ a martingale for any $z$. By the optional stopping theorem, 
\begin{multline*}
\prob_{(x/2, \nu)}\left(\tau_x^H < \sigma(0) \right) \leq     \prob \left(x/2 + \frac{\nu}{\gamma} + \frac{g}{\gamma^2} - B_t - \frac{g}{\gamma}t  \text{  hits  } x \text{  before  } 0  \right)
    = \frac{s\left(x/2 + \frac{\nu}{\gamma} + \frac{g}{\gamma^2}\right) - 1}{s\left(x\right) - 1}  \\
    \leq  \, \exp\left\{- x g/2\gamma \right\},
\end{multline*}
for all $\nu \leq \frac{\gamma x}{4} - \frac{g}{\gamma}$. To prove the lower bound first note that if $h > x$, then $\tau_x^H < \sigma(0)$. For $h \in [(\gamma / 2g)\log(2), \> x]$ and $t \le \sigma(0)$, noting $\nu - V_t \ge 0$ (by \eqref{eqn:increase_on_boundary_only}), we use \eqref{eqn:system_starts_interior} to calculate,
$$
H_t = h + \frac{\nu}{\gamma} -\frac{V_t}{\gamma} - B_t - \frac{g}{\gamma}t \ge (\gamma/2g)\log 2- B_t - \frac{g}{\gamma}t.
$$
The optional stopping theorem again gives,
\begin{multline*}
    \prob_{(h, \nu)}\left(\tau_x^H < \sigma(0)\right) 
    \geq \prob\left((\gamma/2g)\log 2 - B_t - \frac{g}{\gamma}t  \text{  hits  } x \text{  before  } 0 \right)
    = \frac{s\left((\gamma/2g)\log 2\right) - 1}{s\left(x\right) - 1} \\
    = \frac{\exp\left\{-2x \> g/\gamma \right\}}{1 - \exp\left\{-2x \> g/\gamma \right\}}
    \geq \exp\left\{-2x \> g/\gamma\right\}.
\end{multline*}
\end{proof}
\begin{lemma} \label{lemma:gap_hits_zero_before_vel_large}
Fix $a > -g/\gamma$. For any $x>0$,
$$\underset{\nu \in \left[a, \left(a+\frac{g}{\gamma}\right)e^{\gamma^2x/(4g)} - \frac{g}{\gamma}\right]}{\sup} \>\prob_{(x/2, \nu)}\left(\sigma(0) < \tau_a^V  \right) \leq \, \frac{2\sqrt{2\gamma}}{\sqrt{\pi g x}}\exp\left\{-x \> g/8\gamma\right\}.$$
\end{lemma}
\begin{proof}
Observe that for any $x>0$ and any $\nu >  -g/\gamma$, when $(H_0,V_0) = (x/2, \nu)$ we obtain from \eqref{eqn:system},
$$
H_t = H_0 + S_t - B_t + L_t \ge \frac{x}{2} - \frac{g}{\gamma}t - B_t,
$$
where we used $S_t = \int_0^t V_u du \ge -gt/\gamma$ for all $t \ge 0$. From this bound, we conclude that $H_t \ge \frac{x}{4} - B_t$ for all $t \le \frac{\gamma x}{4g}$. Thus, if $B_t < x/4$ for all $t \le \frac{\gamma x}{4g}$, then $\sigma(0) > \frac{\gamma x}{4g}$. In particular, along with \eqref{eqn:system_starts_interior}, this implies that if $t \leq \frac{\gamma x}{4g}$, $V_t = (\nu + g/\gamma ) e^{-\gamma t} - g/\gamma$. The right hand side of this equation equals $a$ when $t = \gamma^{-1}\log\left(\frac{\nu + \frac{g}{\gamma}}{a+\frac{g}{\gamma}}\right)$. If $\nu \in \left[a, \left(a+\frac{g}{\gamma}\right)e^{\gamma^2x/(4g)} - \frac{g}{\gamma}\right]$, then $\gamma^{-1}\log\left(\frac{\nu + \frac{g}{\gamma}}{a+\frac{g}{\gamma}}\right) \le \frac{\gamma x}{4g}$. We conclude that for $x>0$ and $\nu \in \left[a, \left(a+\frac{g}{\gamma}\right)e^{\gamma^2x/(4g)} - \frac{g}{\gamma}\right]$, if $B_t < x/4$ for all $t \le \frac{\gamma x}{4g}$, then $\tau^V_a <\sigma(0)$. Consequently,
$$
\underset{\nu \in \left[a, \left(a+\frac{g}{\gamma}\right)e^{\gamma^2x/(4g)} - \frac{g}{\gamma}\right]}{\sup} \>\prob_{(x/2, \nu)}\left(\sigma(0) < \tau_a^V  \right) \leq \mathbb{P}\left(\sup_{t \le \frac{\gamma x}{4g}}B_t \ge x/4\right) \le \frac{2\sqrt{2\gamma}}{\sqrt{\pi g x}}\exp\left\{-xg/8\gamma\right\}.
$$
\end{proof}

\begin{theorem}  \label{thm:gap_bounds} There exist positive constants $x'(\gamma, g), C, C'$ such that
$$e^{-C'} \exp\left\{- 2xg/\gamma \right\} \leq \prob_{\mathsmaller{\renewalpt}}\left(\tau^H_{x} < \zeta\right) \leq e^{C} \,\, \exp\left\{- xg/16\gamma \right\},$$
for all $x > x'(\gamma, g)$.
\end{theorem}

\begin{proof}  Theorem \ref{thm:velocity_bounds} shows there exist positive constants $x_0(\gamma, g)$, $C, C'$ such that
\begin{equation} \label{eqn:gap_bounds_interim1}
    \prob_{\mathsmaller{\renewalpt}}\left(\tau^V_{\sqrt{x\frac{g}{\gamma(1+\gamma)}}} < \zeta \right) \leq e^C \, e^{-\frac{1+\gamma}{4}\left(\sqrt{x\frac{g}{\gamma(1+\gamma)}} + g/(1+\gamma) \right)^2} \leq  e^{C'} \, e^{-xg/4\gamma},
\end{equation}
for $x > x_0(\gamma, g)$. The union bound then gives,
\begin{multline}\label{eqn:gap_bounds_interim2}
    \prob_{\mathsmaller{\renewalpt}}\left(\tau_x^H < \zeta \right) \leq \prob_{\mathsmaller{\renewalpt}}\left(\tau_x^H < \zeta < \tau^V_{\sqrt{x\frac{g}{\gamma(1+\gamma)}}}  \right) + \prob_{\mathsmaller{\renewalpt}}\left(\tau^V_{\sqrt{x\frac{g}{\gamma(1+\gamma)}}} < \zeta\right)\\
    \leq \prob_{\mathsmaller{\renewalpt}}\left(\tau_x^H < \zeta < \tau^V_{\sqrt{x\frac{g}{\gamma(1+\gamma)}}} \right) + e^C\> e^{-xg/4\gamma}.
\end{multline}
Fix $a = -\left(g + g/2\gamma \right)/(1+\gamma)$ and $b = -\left(g - g/2(1+\gamma) \right)/(1+\gamma)$, as in the definition \eqref{eqn:renewal_def} of $\zeta$. Choose $x_0(\gamma, g)$ large enough that $\left(a+\frac{g}{\gamma}\right)e^{\gamma^2x/(4g)} - \frac{g}{\gamma} > \frac{\gamma x}{4} - \frac{g}{\gamma} > \sqrt{x\frac{g}{\gamma(1+\gamma)}}$ for all $x > x_0(\gamma, g)$. The strong Markov property at $\tau_{x/2}^H$ and Lemmas \ref{lemma:gap_increases_before_zero} and \ref{lemma:gap_hits_zero_before_vel_large} show there exist constants $x_1(\gamma, g) > x_0(\gamma, g)$ and $C > 0$ such that for $x > x_1(\gamma, g)$,
\begin{multline}\label{eqn:gap_bounds_interim3}
    \underset{\nu \in \left[a, 0\right]}{\sup}\>\prob_{(0, \nu)}\left(\tau_x^H \leq \tau_a^V \wedge \tau_{\sqrt{x \frac{g}{\gamma(1+\gamma)}}}^V\right) \leq \underset{\nu \in \left[a, \sqrt{x\frac{g}{\gamma(1+\gamma)}}\right]}{\sup}\>\prob_{(x/2, \nu)}\left(\tau_x^H \leq \tau_a^V\right)\\
    \leq  \underset{\nu \in \left[a, \left(a+\frac{g}{\gamma}\right)e^{\gamma^2x/(4g)} - \frac{g}{\gamma}\right]}{\sup}\> \prob_{(x/2, \nu)}\left(\sigma(0) < \tau_x^H \leq \tau_a^V\right) + \underset{\nu \in \left[a,\frac{\gamma x}{4} - \frac{g}{\gamma}\right]}{\sup}\> \prob_{(x/2, \nu)}\left(\tau_x^H < \sigma(0)\right)\\
    \leq e^C \> e^{-xg/8\gamma}.
\end{multline}
Fix $x > x_1(\gamma, g)$. We define slight modifications of the stopping times given in \eqref{eqn:def_stop_times_below}, and therefore we use the same notation. Define $\alpha_{-1} = 0$ and $\alpha_0 =  \tau_a^V \wedge \tau^V_b$. If $\tau_b^V < \tau_a^V$, define $\alpha_j = \alpha_0$ for all $k \ge 0$ and $N^{-} = 0$. For $k \ge 0$, if $V_{\alpha_{3k}} = a$, 
\begin{eqnarray}\label{eqn:gap_bounds_interim_times}
&&\alpha_{3k + 1} = \inf\left\{t \geq \alpha_{3k}\>|\> H_t \leq x/2\right\}, \nonumber \\
&&\alpha_{3k+2}, \alpha_{3k+3} \quad \text{defined exactly as in } \eqref{eqn:def_stop_times_below}.
\end{eqnarray}
If $V_{\alpha_{3k}} = -g/(1+\gamma)$ then $\alpha_j = \alpha_{3k}$ for all $j \geq 3k$. As before, set $ N^- = \inf \left\{k \ge 1\>|\> V_{\alpha_{3k}} = -g/(1+\gamma) \right\}$. We consider an arbitrary, fixed $x > x_1(\gamma, g)$ and suppress the dependence of $\alpha_1, \alpha_2 \ldots $ on $x$. Since $b < 0 < \sqrt{x \frac{g}{\gamma(1+\gamma)}}$, \eqref{eqn:gap_bounds_interim3} shows, 
\begin{multline}\label{eqn:gap_bounds_interim4}
   \prob_{\mathsmaller{\renewalpt}}\left(\tau_x^H \in (\alpha_{-1}, \alpha_0]\right) = \prob_{\mathsmaller{\renewalpt}}\left(\tau_x^H \leq \tau_a^V \wedge \tau_b^V \right)  \leq \prob_{\mathsmaller{\renewalpt}}\left(\tau_x^H \leq \tau_a^V \wedge \tau_{\sqrt{x \frac{g}{\gamma(1+\gamma)}}}^V\right) \\
    \leq e^C \> e^{-xg/8\gamma}.
\end{multline}
Fix $k \geq 1$. Recall that on $N^- \geq k$, $(H_{\alpha_{3k-1}}, V_{\alpha_{3k-1}}) = \left(0, -(g+g/4\gamma)/(1+\gamma)\right)$. Starting from $\alpha_{3k -1} $,  the velocity cannot rise to $\sqrt{x \frac{g}{\gamma(1+\gamma)}}$, for any $x > 0$, without first passing through $\renewalpt$, by \eqref{eqn:increase_on_boundary_only}. In addition, When $N < k$, $\alpha_{3k-1} = \alpha_{3k}$ and $\tau_x^H \in (\alpha_{3k -1}, \alpha_{3k} ]$ is impossible. We use these observations, along with the strong Markov property at $\alpha_{3k-1}$ and \eqref{eqn:gap_bounds_interim3}, to obtain for any $k \geq 1$,
\begin{multline}\label{eqn:gap_bounds_interim5}
    \prob_{\mathsmaller{\renewalpt}}\left(\tau_x^H \in (\alpha_{3k-1}, \alpha_{3k}]\right) = \prob_{\mathsmaller{\renewalpt}}\left(\tau_x^H \in (\alpha_{3k-1}, \alpha_{3k}], \> N^- \geq k\right)\\ = \expec_{\mathsmaller{\renewalpt}}\left(\mathbbm{1}_{N^- \geq k} \> \prob_{(H_{\alpha_{3k-1}}, V_{\alpha_{3k-1}})}\left(\tau_x^H \leq  \tau_a^V\wedge \tau_{-g/(1+\gamma)}^V\right)\right)\\
    \leq  \underset{\nu \in \left[a, 0\right]}{\sup}\>\prob_{(0, \nu)}\left(\tau_x^H \leq \tau_a^V\wedge \tau_{-g/(1+\gamma)}^V \right)\>\prob_{\mathsmaller{\renewalpt}}\left( N^- \geq k\right)  \leq e^C \> e^{-xg/8\gamma}\>\prob_{\mathsmaller{\renewalpt}}\left( N^- \geq k\right).
\end{multline}
Now we estimate the probability of $\tau_x^H \in (\alpha_{3k}, \alpha_{3k + 1}]$ for $k \geq 0$. If $H_{\alpha_{3k}} \leq x/2$ then $\alpha_{3k + 1} = \alpha_{3k}$ and $\tau_x^H \in (\alpha_{3k}, \alpha_{3k + 1}]$ is impossible. Therefore, we need only consider cases in which $H$ reaches $x/2$ between $\alpha_{3k - 1}$, when $H$ is zero, and $\alpha_{3k}$. Therefore, once again we apply the strong Markov property at $\alpha_{3k-1}$ and \eqref{eqn:gap_bounds_interim3}, with $x/2$ in place of $x$, to obtain
\begin{multline}\label{eqn:gap_bounds_interim6}
    \prob_{\mathsmaller{\renewalpt}}\left(\tau_x^H \in (\alpha_{3k}, \alpha_{3k+1}]\right)\\
    = \prob_{\mathsmaller{\renewalpt}}\left(\tau_x^H \in (\alpha_{3k}, \alpha_{3k+1}], \> \inf\{t \geq \alpha_{3k-1} \> | \> H_t \geq x/2 \} \in (\alpha_{3k-1}, \alpha_{3k}],\> N^- \geq k\right)\\
    \leq \underset{\nu \in \left[a, 0\right]}{\sup}\>\prob_{(0, \nu)}\left(\tau_{x/2}^H \leq \tau_a^V \wedge \tau_{-g/(1+\gamma)}^V \right)\>\prob_{\mathsmaller{\renewalpt}}\left( N^- \geq k\right)
    \leq e^C \> e^{-xg/16\gamma}\>\prob_{\mathsmaller{\renewalpt}}\left( N^- \geq k\right).
\end{multline}
Note that for any $k \geq 0$,  the fact \eqref{eqn:increase_on_boundary_only} that the velocity increases only where $H = 0$ and the definition of $\alpha_{3k + 1}$ show $(H_{\alpha_{3k+1}}, V_{\alpha_{3k+1}}) \in[0, x/2] \times \left(-g/\gamma, a\right]$. Suppose we have initial conditions such that $(H_0, V_0) \in [0, x/2] \times \left(-g/\gamma, a\right]$. When $t < \tau_{-(g + g/4\gamma)/(1+\gamma)}^V$, noting $a < -(g + g/4\gamma)/(1+\gamma)$, we use \eqref{eqn:system}, $S_t \leq -t(g + g/4\gamma)/(1+\gamma)$ and $-V_t > 0$ to show
\begin{equation}\label{eqn:gap_bounds_interim7}
    H_t \leq H_t - V_t = H_0 - V_0 + (1+\gamma)S_t + tg - B_t \leq x/2 + g/\gamma -tg/4\gamma - B_t.
\end{equation}
By \eqref{eqn:gap_bounds_interim7}, if $H$ hits $x$ before $V$ hits $-(g + g/4\gamma)/(1+\gamma)$, then $\sup_{t < \infty} \left(-B_t - tg/4\gamma \right)$ must have reached $x/2 - g/\gamma$, which is positive so long as we have chosen $x_1(\gamma, g)$ large enough. We use the strong Markov property at $\alpha_{3k+1}$, \eqref{eqn:gap_bounds_interim7} and $\underset{t < \infty}{\sup}\left(-B_t - tg/4\gamma \right) \overset{d}{=} \text{Exponential}(g/2\gamma)$ (see Chapter 3.5 of \cite{kar}) to show for any $k \geq 0$,
\begin{multline}\label{eqn:gap_bounds_interim8}
    \prob_{\mathsmaller{\renewalpt}}\left(\tau_x^H \in (\alpha_{3k+1}, \alpha_{3k+2}]\right) =\prob_{\mathsmaller{\renewalpt}}\left(\tau_x^H \in (\alpha_{3k+1}, \alpha_{3k+2}], \> N^- \geq k+1\right) \\
    \leq \underset{(h, \nu) \in [0, x/2] \times \left(-g/\gamma, a\right]}{\sup}\>\prob_{(h, \nu)}\left(\tau_x^H < \tau_{-(g + g/4\gamma)/(1+\gamma)}^V \right) \>\prob_{\mathsmaller{\renewalpt}}\left( N^- \geq k+1\right) \\
    \leq \prob \left(\underset{t < \infty}{\sup}\left(-B_t - tg/4\gamma \right) \geq x/2 - g/\gamma \right)\>\prob_{\mathsmaller{\renewalpt}}\left( N^- \geq k\right)
    = e^{-(g/2\gamma)(x/2 - g/\gamma)}\>\prob_{\mathsmaller{\renewalpt}}\left( N^- \geq k\right)\\
    = e^C \> e^{-xg/4\gamma}\>\prob_{\mathsmaller{\renewalpt}}\left( N^- \geq k\right).
\end{multline}
Recall that when $\tau_a^V < \tau_b^V$, we have $\zeta = \alpha_{3N^-}$. Therefore, on $\tau_a^V < \tau_b^V$ we have $\tau_x^H < \zeta$ if and only if $\tau_x^H \in (\alpha_{3k + j}, \alpha_{3k+j + 1}]$ for some $k \geq 0$, $j = -1, 0, 1$. As a result, we use \eqref{eqn:gap_bounds_interim4}, \eqref{eqn:gap_bounds_interim5}, \eqref{eqn:gap_bounds_interim6} and \eqref{eqn:gap_bounds_interim8} to show \begin{multline}\label{eqn:gap_bounds_interim10}
    \prob_{\mathsmaller{\renewalpt}}\left(\tau_x^H < \zeta, \> \tau_a^V < \tau_b^V \right) = \sum_{k = 0}^\infty \sum_{j = -1}^1 \prob_{\mathsmaller{\renewalpt}}\left(\tau_x^H \in (\alpha_{3k+j}, \alpha_{3k+j+1}]\right) \\
    \leq 3e^C \> e^{-xg/16\gamma}\>\sum_{k = 0}^\infty\prob_{\mathsmaller{\renewalpt}}\left( N^- \geq k\right).
\end{multline}
The exact same argument as in Lemma \ref{lemma:stop_renewal_from_below} shows $\expec_{\mathsmaller{\renewalpt}}\> N^- \leq e^{C'}.$ Recalling that $\tau_a^V < \tau_b^V$ implies $\zeta < \tau_{\sqrt{x\frac{g}{\gamma(1+\gamma)}}}^V$,  we have by \eqref{eqn:gap_bounds_interim10} a positive constant $C''$ such that,
\begin{equation}\label{eqn:gap_bounds_interim11}
    \prob_{\mathsmaller{\renewalpt}}\left(\tau_x^H < \zeta < \tau^V_{\sqrt{x\frac{g}{\gamma(1+\gamma)}}} , \> \tau_a^V < \tau_b^V \right)  = \prob_{\mathsmaller{\renewalpt}}\left(\tau_x^H < \zeta, \> \tau_a^V < \tau_b^V \right) \leq e^{C''}\> e^{-xg/16\gamma}.
\end{equation}
Now consider the case $\tau_b^V < \tau_a^V$. Once again fix $x > x_1(\gamma, g)$. \eqref{eqn:gap_bounds_interim3} directly implies
\begin{multline}\label{eqn:gap_bounds_interim12}
    \prob_{\mathsmaller{\renewalpt}}\left(\tau_x^H \leq \tau_b^V < \tau_a^V \right) \leq \prob_{\mathsmaller{\renewalpt}}\left(\tau_x^H \leq \tau_a^V \wedge \tau^V_{\sqrt{x\frac{g}{\gamma(1+\gamma)}}}\right) \leq e^C \> e^{-xg/8\gamma}.
\end{multline}
We now control $H$ in the time between $\tau_b^V$ and $\zeta$. There are two possibilities: Either $\zeta < \tau_a^V$, or $V$ crosses down to $a$ before the renewal time is reached. In the former case: Since $b \in [a, 0]$ and $H_{\tau_b^V} = 0$ by \eqref{eqn:increase_on_boundary_only}, we use the strong Markov property at $\tau_b^V$ and \eqref{eqn:gap_bounds_interim3} again to show
\begin{multline}\label{eqn:gap_bounds_interim13}
    \prob_{\mathsmaller{\renewalpt}}\left(\tau_x^H \in (\tau_b^V, \zeta], \> \zeta < \tau_a^V \wedge \tau^V_{\sqrt{x\frac{g}{\gamma(1+\gamma)}}} \right) \leq \prob_{(0, b)}\left(\tau_x^H \leq \tau_a^V \wedge \tau^V_{\sqrt{x\frac{g}{\gamma(1+\gamma)}}} \right) \leq e^C \> e^{-xg/8\gamma}.
\end{multline}
In the case where $\tau_a^V < \zeta$, we modify the analysis used to prove \eqref{eqn:gap_bounds_interim11}, as follows. Define $\tilde{\beta}_{-1}=\tau_a^V\wedge \tau_b^V$ and $\tilde{\beta}_0 = \inf\{t \geq \tilde{\beta}_{-1}\>|\> V_t = a \}$. Define $\{\tilde{\beta}_j\}_{j \ge 1}$ and $\tilde{N}^-$ analogously to  $\{\alpha_{j}\}_{j \ge 1}$ and $N^-$ in \eqref{eqn:gap_bounds_interim_times}. The strong Markov property at $\tilde{\beta}_{-1}$ and \eqref{eqn:gap_bounds_interim3} show,
\begin{multline}\label{eqn:gap_bounds_interim14}
    \prob_{\mathsmaller{\renewalpt}}\left(\tau_x^H \in (\tilde{\beta}_{-1}, \tilde{\beta}_0], \> \tau_a^V < \zeta < \tau^V_{\sqrt{x\frac{g}{\gamma(1+\gamma)}}} \right) \leq \prob_{(0, b)}\left(\tau_x^H \leq \tau_a^V \wedge \tau^V_{\sqrt{x\frac{g}{\gamma(1+\gamma)}}} \right) \leq e^C \> e^{-xg/8\gamma}.
\end{multline}
The analysis of \eqref{eqn:gap_bounds_interim4}, \eqref{eqn:gap_bounds_interim5}, \eqref{eqn:gap_bounds_interim6} and \eqref{eqn:gap_bounds_interim8} is now repeated, with $\tilde{\beta}_j$ in place of $\alpha_j$ for $j \ge 0$ and $\tilde{N}^-$ in place of $N^-$, giving for $k \ge 0$ and $j = 0, 1, 2$,
\begin{equation}\label{eqn:gap_bounds_interim15}
    \prob_{\mathsmaller{\renewalpt}}\left(\tau_x^H \in (\tilde{\beta}_{3k+j}, \tilde{\beta}_{3k+j+1}], \> \tau_a^V < \zeta < \tau^V_{\sqrt{x\frac{g}{\gamma(1+\gamma)}}} \right) 
    \leq e^C \> e^{-xg/16\gamma}\>\prob_{\mathsmaller{\renewalpt}}\left( \tilde{N}^- \geq k\right).
\end{equation}
Combining \eqref{eqn:gap_bounds_interim12}, \eqref{eqn:gap_bounds_interim13}, \eqref{eqn:gap_bounds_interim14} and \eqref{eqn:gap_bounds_interim15},
\begin{multline}\label{eqn:gap_bounds_interim16}
    \prob_{\mathsmaller{\renewalpt}}\left(\tau_x^H < \zeta < \tau^V_{\sqrt{x\frac{g}{\gamma(1+\gamma)}}}, \> \tau_b^V < \tau_a^V \right) \leq \prob_{\mathsmaller{\renewalpt}}\left(\tau_x^H \in (0, \tau_b^V], \> \zeta < \tau^V_{\sqrt{x\frac{g}{\gamma(1+\gamma)}}}, \> \tau_b^V < \tau_a^V \right) \\
    + \prob_{\mathsmaller{\renewalpt}}\left(\tau_x^H \in (\tau_b^V, \zeta], \> \zeta < \tau_a^V \wedge \tau^V_{\sqrt{x\frac{g}{\gamma(1+\gamma)}}} \right)
    + \prob_{\mathsmaller{\renewalpt}}\left(\tau_x^H \in (\tau_b^V, \zeta], \> \tau_a^V < \zeta < \tau^V_{\sqrt{x\frac{g}{\gamma(1+\gamma)}}} \right)\\
    \leq 2e^C \> e^{-xg/8\gamma} + \sum_{k = 0}^\infty \sum_{j = -1}^1 \prob_{\mathsmaller{\renewalpt}}\left(\tau_x^H \in (\tilde{\beta}_{3k+j}, \tilde{\beta}_{3k+j+1}]\right)\\
    \leq 2e^C \> e^{-xg/8\gamma} + e^C \> e^{-xg/16\gamma} \sum_{k = 0}^\infty \prob_{\mathsmaller{\renewalpt}}\left(\tilde{N}^- \geq k\right).
\end{multline}
Arguing as in \eqref{eqn:gap_bounds_interim11}, we achieve,
\begin{equation}\label{eqn:gap_bounds_interim17}
    \prob_{\mathsmaller{\renewalpt}}\left(\tau_x^H < \zeta < \tau^V_{\sqrt{x\frac{g}{\gamma(1+\gamma)}}}, \> \tau_b^V < \tau_a^V \right) \leq e^{C''}\> e^{-xg/16\gamma}.
\end{equation}
Our choice of $x > x_1(\gamma, g)$ in \eqref{eqn:gap_bounds_interim2}, \eqref{eqn:gap_bounds_interim11} and \eqref{eqn:gap_bounds_interim17} was arbitrary, so the upper bound of the theorem is proven.

We now prove the lower bound. For any $x > 0$, we consider a path in which $H$ attains a positive value $h_0$ before the velocity leaves $[a, b]$, then rises to $x$ before returning zero. Since by definition \eqref{eqn:renewal_def}, $\zeta> \tau_a^V \wedge \tau_b^V$ and $H_{\zeta} = 0$, this implies $\tau_x^H < \zeta$. We select any $h_0 \geq (\gamma/2g)\log 2$ and $x > h_0$. The strong Markov property at $\tau_{h_0}^H$ and Lemma \ref{lemma:gap_increases_before_zero} give,
\begin{multline}\label{eqn:gap_bounds_interim18}
    \prob_{\mathsmaller{\renewalpt}}\left(\tau_x^H < \zeta \right) \ge  \prob_{\mathsmaller{\renewalpt}}\left(\tau_{h_0}^H < \tau_a^V \wedge \tau_b^V, \> \tau_{h_0}^H < \tau_x^H < \sigma(\tau_{h_0}^H) \right) \\
    \ge   \prob_{\mathsmaller{\renewalpt}}\left(\tau_{h_0}^H < \tau_a^V \wedge \tau_b^V\right) \> \underset{\nu \in (a, b)}{\inf}\>\prob_{(h_0, \nu)}\left(\tau_x^H < \sigma(0) \right)
    \geq e^{-C} \> e^{-2xg/\gamma},
\end{multline}
where $C = - \log \left(\prob_{\mathsmaller{\renewalpt}}\left(\tau_{h_0}^H < \tau_a^V \wedge \tau_b^V\right)\right)$. It remains only to show $C < \infty.$ Suppose $H_0 = 0$ and $V_0 = -g/(1+\gamma)$. Define $T = \frac{1}{\gamma}\log \left(\frac{-g/(1+\gamma) + g/\gamma}{a + g/\gamma} \right) = \frac{1}{\gamma} \log2 $. By \eqref{eqn:hitting_time_velocity}, $V$ cannot hit $a$ before time $T$ and hence, $S_u \geq ua$ for $u < T$. Thus, if $\inf_{u < T} \> B_u < -(h_0 - T a) < 0$ (recalling $a < 0$), we obtain from system equations \eqref{eqn:system}
$$
\sup_{u < T}\>H_u = \sup_{u<T}(S_u-B_u + L_u) \geq \sup_{u < T}\> \left(ua - B_u \right) > h_0.
$$
If, in addition, $\sup_{u < T} \> \left(B_u -ua\right) < b + g/(1+\gamma)$ then $V_u \leq -g/(1+\gamma) - u(\gamma a + g) + \sup_{s < T} \> \left(B_s - sa \right) < b$ for all $u < T$, since $\gamma a + g > 0$. Recalling that $b + g/(1+\gamma) > 0$, we have shown
\begin{multline}\label{eqn:gap_bounds_interim19}
    \prob_{\mathsmaller{\renewalpt}}\left(\tau_{h_0}^H < \tau_a^V \wedge \tau_b^V\right) \geq \prob \left(\inf_{u < T} \> B_u < -(h_0 - T a) , \>  \sup_{u < T} \> \left(B_u -ua\right) < b + g/(1+\gamma)\right) > 0.
\end{multline}
\eqref{eqn:gap_bounds_interim19} shows $C < \infty$ in \eqref{eqn:gap_bounds_interim18}, and the theorem is proven with $x'(\gamma, g) = x_1(\gamma, g) \vee \left[(\gamma/2g)\log 2\right]$.
\end{proof}

\section{Tail bounds for $\pi$ and path fluctuations}\label{osctail}

This section is devoted to the proof of Theorems \ref{thm:tails_under_stationarity} and \ref{thm:fluctuations}.

\begin{proof}[Proof of Theorem \ref{thm:tails_under_stationarity}]
First we prove the theorem's lower bound for $\pi\left(\Real_+ \times (y, \infty) \right)$. Fix $y'(\gamma, g)$ as in Theorem \ref{thm:velocity_bounds} and $y > y'(\gamma, g) > 0$. Recall the notation $\tau^V_z(\alpha) = \inf\{t \geq \alpha \> : \> V_t = z \}$ for any stopping time $\alpha$. On the set $\tau_{2y}^V < \zeta$, we have $\tau_{y}^V(\tau_{2y}^V) < \zeta$. Theorem \ref{thm:velocity_bounds}, the strong Markov property at $\tau_{y}^V$ and the definition of $\pi$ in Theorem \ref{thm:statmeas_renewal} show there exists a constant $C > 0$ such that,
\begin{multline}\label{eqn:pi_bounds1}
    \expec_{\mathsmaller{\renewalpt}} \> \left(\zeta \right) \> \pi\left(\Real_+ \times (y, \infty) \right) = \expec_{\mathsmaller{\renewalpt}}\left(\int_0^\zeta \, \indi{ V_t > y} \, dt\right) \\
    \ge \expec_{\mathsmaller{\renewalpt}}  \left(\indi{\tau_{2y}^V < \zeta} \int_{\tau_{2y}^V}^{\tau_{y}^V(\tau_{2y}^V)} \, \indi{ V_t > y} \, dt\right)
    = \prob_{\mathsmaller{\renewalpt}}\left(\tau_{2y}^V < \zeta \right) \> \expec_{(0, 2y)}\left(\int_0^{\tau_{y}^V}\, \indi{ V_t > y} \, dt\right) \\
    \ge e^{-C} \> e^{-2(1+\gamma)(2y + g/(1+\gamma))^2} \> \expec_{(0, 2y)} \> \tau_{y}^V \ge e^{-C} \> e^{-4(1+\gamma)(y + g/(1+\gamma))^2} \> \expec_{(0, 2y)} \> \tau_{y}^V.
\end{multline}
By \eqref{eqn:hitting_time_velocity}, $\expec_{(0, 2y)} \> \tau_{y}^V \ge \frac{1}{\gamma}\log\left(\frac{2y + g/\gamma}{y + g/\gamma} \right) \ge \frac{1}{\gamma}\log\left(\frac{2y'(\gamma, g) + g/\gamma}{y'(\gamma, g) + g/\gamma} \right) > 0$ for all $y > y'(\gamma, g)$. So \eqref{eqn:pi_bounds1} proves the lower bound of the theorem for $\pi\left(\Real_+ \times (y, \infty) \right)$, for all $y > y'(\gamma, g)$.

We now prove the lower bound for $\pi\left((x, \infty) \times (-g/\gamma, \infty) \right)$. Fix $x'(\gamma, g) > 0$ as in Theorem \ref{thm:gap_bounds} and $x > x'(\gamma, g)$. Proceeding similarly to \eqref{eqn:pi_bounds1}, by Theorem \ref{thm:gap_bounds} there exists a $C' > 0$ such that,
\begin{equation}\label{eqn:pi_bounds2}
    \expec_{\mathsmaller{\renewalpt}} \> \left(\zeta \right) \> \pi\left((x, \infty) \times (-g/\gamma, \infty) \right)
    \ge e^{-C'} \> e^{-4xg/\gamma} \> \underset{\nu > -g/\gamma}{\inf}\> \expec_{(2x, \nu)} \> \tau_{x}^H.
\end{equation}
When $H_0 = 2x$, \eqref{eqn:system} and \eqref{eqn:velocity_lower_bounded} show $H_t \ge 2x + S_t - B_t \ge 2x - tg/\gamma - B_t$ for any $V_0 > -g/\gamma$. Therefore, $\tau_{x}^H \ge \inf\{t \ge 0 \> : \> -B_t - tg/\gamma = -x \}$ for any initial condition $V_0$. The expected hitting time of Brownian motion with drift $-g/\gamma$ at level $-x'(\gamma,g)$ is strictly positive and finite (e.g. Ch. 3C in \cite{kar}), so we have $C''>0$ such that $\inf_{\nu > -g/\gamma}\> \expec_{(2x, \nu)} \> \tau_{x}^H \ge C''$ for each $x > x'(\gamma, g)$. This fact and \eqref{eqn:pi_bounds2} prove the required lower bound for $\pi\left((x, \infty) \times (-g/\gamma, \infty) \right)$.

We now show the upper bounds of the theorem. Again using the representation for $\pi$ in Theorem \ref{thm:statmeas_renewal} and the velocity bounds in Theorem \ref{thm:velocity_bounds}, we obtain for $y > y'(\gamma,g)$,
\begin{multline}\label{eqn:pi_bounds3}
 \expec_{\mathsmaller{\renewalpt}} \> \left(\zeta \right) \> \pi\left(\Real_+ \times (y, \infty) \right) = \expec_{\mathsmaller{\renewalpt}}\left(\int_0^\zeta \, \indi{ V_t > y} \, dt\right) \le \expec_{\mathsmaller{\renewalpt}}\left(\indi{\tau_{y}^V < \zeta}\left(\zeta - \tau_{y}^V\right) \right)\\
   \le \sqrt{ \prob_{\mathsmaller{\renewalpt}}\left(\tau_{y}^V < \zeta \right)}\sqrt{\expec_{\mathsmaller{\renewalpt}}(\zeta^2)}
    \le e^{c/2} \, e^{-\frac{1+\gamma}{8}(y+g/(1+\gamma))^2}\sqrt{\expec_{\mathsmaller{\renewalpt}}(\zeta^2)}.
\end{multline}
The upper bound for $\pi\left(\Real_+ \times (y, \infty) \right)$ follows from \eqref{eqn:pi_bounds3} upon noting that $\expec_{\mathsmaller{\renewalpt}}(\zeta^2)<\infty$, which is a consequence of Theorem \ref{thm:renewal}. The upper bound for $\pi\left((x, \infty) \times (-g/\gamma, \infty) \right)$ is proven similarly using Theorem \ref{thm:gap_bounds}.
\end{proof}

\begin{proof}[Proof of Theorem \ref{thm:fluctuations}] The argument is identical to the one provided for Theorem 2.2 of \cite{grav} and Theorem 2.2 of \cite{queue}, so we give it cursory treatment. To demonstrate the fluctuation result for $V$: Theorem \ref{thm:velocity_bounds}, \eqref{eqn:renewal_sequence} and the Borel-Cantelli lemmas give for any $\epsilon \in (0,1)$,
$$
\frac{\sqrt{1 - \epsilon}}{\sqrt{2}\sqrt{1+\gamma}} \leq \limsup_{n \to \infty} \frac{\sup_{t \in [\zeta_n, \zeta_{n+1}]}\,\, V_t}{\sqrt{\log n}} \leq 2 \frac{\sqrt{1+\epsilon}}{\sqrt{1+\gamma}}, \quad \ \text{almost surely.}
$$
A sub-sequence argument, and the observation $\lim_{n \rightarrow \infty} \frac{\zeta_n}{n} = \expec \zeta_0$ almost surely, complete the proof. The second statement is proven similarly, with Theorem \ref{thm:gap_bounds}.
\end{proof}
\begin{proof}[Proof of Theorem \ref{thm:lln}] By Theorem \ref{thm:fluctuations},
\begin{equation}
   0 = \lim_{t \to \infty} \frac{H_t - V_t}{t} = \lim_{t \to \infty} \frac{H_0 - V_0 + \left(1 + \gamma\right)S_t  - B_t + gt }{t},
\end{equation}
and thus, using $\frac{B_t}{t} \rightarrow 0$ almost surely as $t \rightarrow \infty$, we obtain $\frac{S_t}{t} \to -\frac{g}{1+\gamma}$. In addition, again by Theorem \ref{thm:fluctuations}, $\lim_{t \to \infty} \frac{H_t}{t} = \lim_{t \to \infty} \frac{S_t - X_t}{t} = 0$, giving the result.
\end{proof}
\section{Exponential ergodicity}\label{geoerg}
This section will prove Theorem \ref{thm:expo_ergodic}. We first show the process is ergodic, in the sense that $P^t((h, \nu, \> \cdot \>)$ converges to $\pi$ in total variation, using coupling techniques of Ch. 10 \cite{thor}. We show this can be upgraded to exponentially fast convergence using Lyapunov function techniques. 

Typical proofs of exponential ergodicity via Harris' Theorem available in the literature (e.g. \cite{hairandmat,meyntweedie,mattinglypillai,cooke,budhirajalee}) rely on the existence of continuous densities of $P^t$ with respect to Lebesgue measure. Such densities give a positive chance of coupling two versions of the process within a set toward which the process has a strong drift. 

In our case the transition laws do not have densities, as can be verified by observing the velocity decreases deterministically away from the boundary $\partial S$. In addition, the generator of the process is \emph{not hypoelliptic} in the interior of the domain, which makes the situation more complicated: Hypoellipticity is a standard tool for establishing exponential ergodicity in the absence of ellipticity (\cite{mattingly2002geometric}). We adapt techniques from \cite{thor}, which involve coupling the renewal times of two versions of the process, to furnish exponential ergodicity for our model.

Before we proceed to details, we provide a proof outline. The first step is proving convergence to stationarity in total variation distance. This is done using Theorem 3.3, Ch. 10 of \cite{thor} after proving in Lemma \ref{lemma:spread_out} that the gaps between the renewal times are `spread-out' (see Section 3.5, Ch. 10 of \cite{thor}). The total variation convergence result is stated in Theorem \ref{thm:ergodic}. 

This convergence is then upgraded to exponential ergodicity by establishing the \emph{drift condition} \eqref{eqn:drift_cond} and \emph{minorization condition} \eqref{eqn:minorization_cond} which can be used to furnish exponential ergodicity using the recipe in \cite{ergodicity}. The drift condition is established via Lyapunov functions that can be obtained from exponential moments of hitting times of certain carefully chosen sets. The finiteness of exponential moments for one such set $\Lambda$ is shown in Lemma \ref{lemma:lyapunov_exponential_moment}. 

The minorization condition is obtained by establishing certain structural properties for the Markov process along the lines of \cite{stability}, which are stated in Lemma \ref{lemma:structural}. In particular, we show that the process $(H_t,V_t)_{t \ge 0}$ is $\pi$-irreducible and positive Harris recurrent and the aforementioned set $\Lambda$ is `petite'. These structural properties imply a stronger version of \eqref{eqn:minorization_cond} as stated in Lemma \ref{lemma:minorization}. Finally, the proof of Theorem \ref{thm:expo_ergodic} is completed at the end of the Section.
\begin{lemma} \label{lemma:spread_out} There exists a non-negative function $f$ with $\int_0^{\infty}f(x)\,dx>0$ such that for every measurable set $A \subset [0, \infty)$, $$\prob_{\mathsmaller{\renewalpt}}\left(\zeta  \in A\right) \geq \int_A f(x) \, dx. $$
\end{lemma}
\begin{proof} Recall $a, b$ in the definition \eqref{eqn:renewal_def} of $\zeta$, where $ -g/\gamma < a < -g/(1+\gamma) < b < 0$. By definition of $\zeta$ and \eqref{eqn:increase_on_boundary_only}, when $\tau_a^V < \tau_b^V$ we have $\zeta = \tau_a^V + \inf\{t \geq 0 \> : \> V_{\tau_a^V + t} = -g/(1+\gamma)\}$. For any measurable $A \subset [0, \infty)$, we apply the strong Markov property at $\tau_a^V$ to obtain,
\begin{equation}\label{eqn:spread_out0}
    \prob_{\mathsmaller{\renewalpt}}\left(\zeta  \in A\right) \ge \prob_{\mathsmaller{\renewalpt}}\left(\zeta  \in A, \> \tau_a^V < \tau_b^V \right) = \expec_{\mathsmaller{\renewalpt}}\left(\mathbbm{1}_{\tau_a^V < \tau_b^V} \> F_A\left(H_{\tau_a^V}, \tau_a^V\right)\right),
\end{equation}
where $F_A(h, t) = \prob_{(h, a)}\left(\tau_{-g/(1+\gamma)}^V  \in A-t\right)$. We will show the right-hand side of \eqref{eqn:spread_out0} has a density with respect to Lebesgue measure $\lambda$ on $[0, \infty)$. By the Radon-Nikodym Theorem, it suffices to show that for each $u > 0$ and $A \subset [0, u]$,
\begin{equation}\label{eqn:spread_out1}
    \lambda(A) = 0 \implies \expec_{\mathsmaller{\renewalpt}}\left(\mathbbm{1}_{\tau_a^V < \tau_b^V} \> F_A\left(H_{\tau_a^V}, \tau_a^V\right)\right) = 0.
\end{equation}
Fix such a $u$ and $A$, $h\ge0$, and set $(H_0, V_0) = (h, a)$. \eqref{eqn:increase_on_boundary_only} says that if $t$ is a point of increase of $V_t$ it must be a point of increase for $L_t = L_t^{(h, a)} = 0 \vee \sup_{s\leq t}\left(-h +B_s - S_s \right)$. As a result,
\begin{equation}\label{eqn:spread_out2}
    \tau_{-g/(1+\gamma)}^V = \inf\{t \geq 0 \> : \> B_t -(1+\gamma) S_t - gt = h -g/(1+\gamma) - a \} = \tau^{W}_{h -g/(1+\gamma) - a},
\end{equation}
where $W_t = B_t - \int_0^t(1+\gamma) V_{s\wedge \tau_0^V} -gt$. The second equality in \eqref{eqn:spread_out2} follows from the fact that $-g/(1+\gamma) < 0$ and $V_0 = a < -g/(1+\gamma)$, so $W_t = B_t -(1+\gamma) S_t - gt$ for $t \leq \tau_{-g/(1+\gamma)}^V$. Since $t \mapsto V_{t \wedge \tau_0^V}$ is bounded, the Novikov condition and Girsanov's theorem (\cite{kar} Ch. 3.5) show the law of the process $W$ is equivalent to that of standard Brownian motion on any bounded time interval $[0, u]$. Therefore, noting that $h -g/(1+\gamma) - a>0$, $\tau_{h -g/(1+\gamma) - a}^W$ has a density with respect to Lebesgue measure (\cite{kar} Ch. 2.8). Now \eqref{eqn:spread_out2} implies that whenever $\lambda(A) = 0$, $F_A(h, t) = 0$ for any $t \in [0, u]$ and $h \geq 0$. This proves \eqref{eqn:spread_out1}. $f$ can thus be taken as the density with respect to $\lambda$ of the measure $A \mapsto \expec_{\mathsmaller{\renewalpt}}\left(\mathbbm{1}_{\tau_a^V < \tau_b^V} \> F_A\left(H_{\tau_a^V}, \tau_a^V\right)\right)$. Since $\mathbb{P}_{\renewalpt}\left(\tau_a^V < \tau_b^V\right) >0$, we have $\int_0^{\infty}f(x)\,dx>0$. The lemma follows. 
\end{proof}

\begin{theorem} \label{thm:ergodic} For every initial condition $(h, \nu) \in S$,
$$\left\|P^t((h, \nu), \> \cdot \>) - \pi \right\|_{TV} \longrightarrow 0 \quad \text{ as }\> t \to \infty.$$
\end{theorem}

\begin{proof}
This is a consequence of Lemma \ref{lemma:spread_out} and Theorem 3.3, Ch. 10 of \cite{thor}. Using the terminology from that reference: The process $(H, V)$ is classical regenerative by the strong Markov property and \eqref{eqn:renewal_def}, \eqref{eqn:renewal_sequence}. Since $\prob_{(h, \nu)}\left(\zeta_1 - \zeta_0 \in \cdot\right) = \prob_{\mathsmaller{\renewalpt}}\left(\zeta \in \cdot\right)$, Lemma \ref{lemma:spread_out} shows the inter-regeneration time $\zeta_1 - \zeta_0$ is `spread out.' The proof is completed by Theorem 3.3 (b), Ch. 10 of \cite{thor}.
\end{proof}
The following Lemma establishes several structural properties for the process $(H_t,V_t)_{t \ge 0}$. We refer the reader to \cite{stability} for definitions of resolvent kernels, $\pi$-irreducibility, `petite' sets and Harris recurrence that appear in the following Lemma.
\begin{lemma} \label{lemma:structural} Define the (discrete time) resolvent transition kernel $R\left((h, \nu), \, \cdot \,\right) \coloneqq \int_0^\infty e^{-t}P^t\left((h, \nu),\>\cdot \>\right)\, dt$, and the set $\Lambda = [0, 1]\times\left[-\frac{g + g/2(1+\gamma)}{1+\gamma}, \> g/\gamma\right]$. The following hold:
\begin{enumerate}
    \item[(a)] For any measurable set $A$, $\pi(A) > 0$ implies $R((h, \nu), A) > 0$ for any $(h, \nu) \in S$.
    \item[(b)] For any measurable set $A$, $\pi(A) > 0$ implies $\int_0^\infty P^t\left((h, \nu),\>\cdot \>\right)\, dt > 0$ for any $(h, \nu) \in S$. In other words, the process is $\pi$-irreducible.
    \item[(c)] There exists an $\alpha > 0$ and a non-trivial measure $\mu$, equivalent to $\pi$, such that $R((h, \nu), A) \geq \alpha \, \mu(A)$ for all $(h, \nu) \in \Lambda$ and all measurable $A$. In particular, $\Lambda$ is a `petite' set.
    \item[(d)] The process $(H_t,V_t)_{t \ge 0}$ is Harris recurrent, and since the invariant measure $\pi$ is finite the process is positive Harris recurrent. 
\end{enumerate}
\end{lemma} 

\begin{proof}
Using the definition of successive renewal times $\zeta := \zeta_0$ and $\{\zeta_n\}_{n \ge 0}$ given in \eqref{eqn:renewal_def}, we apply the strong Markov property at $\zeta_n, n \geq 0$ to show for any measurable set $A$, with $\F_{\zeta_n}$ denoting the stopped filtration for $(H, V)$, 
\begin{multline}\label{eqn:resolvent_structural0}
    \expec_{(h, \nu)} \> \left(\int_{\zeta_n}^{\zeta_{n+1}} e^{-t} \indi{(H_t, V_t ) \in A}\, dt \> | \> \F_{\zeta_n} \right) = \int_{\zeta_n}^\infty e^{-t} \> \mathbb{E}_{\mathsmaller{\renewalpt}} \> \left(\indi{\zeta > t - \zeta_n, (H_{t-\zeta_n}, V_{t-\zeta_n} ) \in A} \right)\, dt \\
    = e^{-\zeta_n}\>\expec_{\mathsmaller{\renewalpt}} \> \left(\int_0^{\zeta} e^{-t} \indi{(H_t, V_t ) \in A}\, dt \right).
\end{multline}
Define $\mu(A) \coloneqq \expec_{\renewalpt} \left(\int_{0}^{\zeta} e^{-t} \indi{(H_t, V_t ) \in A}\, dt\right)$. Taking expectations in \eqref{eqn:resolvent_structural0} and summing over $n$ we have,
\begin{multline}\label{eqn:resolvent_structural}
    R\left((h, \nu), \, A \,\right) = \expec_{(h, \nu)} \left( \> \int_0^{\zeta} e^{-t} \indi{(H_t, V_t ) \in A}\, dt\right) + \sum_{n=0}^\infty \expec_{(h, \nu)} \> \int_{\zeta_n}^{\zeta_{n+1}} e^{-t} \indi{(H_t, V_t ) \in A}\, dt \\ 
    = \expec_{(h, \nu)} \> \left(\int_0^{\zeta} e^{-t} \indi{(H_t, V_t ) \in A}\, dt\right) + \mu(A) \> \sum_{n=0}^\infty \expec_{(h, \nu)} \> e^{-\zeta_n}.
\end{multline}
The form of $\pi$ given in Theorem \ref{thm:statmeas_renewal} shows the measures $\pi$ and $\mu$ are equivalent. Thus, \eqref{eqn:resolvent_structural} proves (a), and (b) is an immediate consequence of (a).

The arguments used to prove Theorem \ref{thm:renewal} are valid for any starting point $(h, \nu)$ other than $\renewalpt$: Since the process started from $(h, \nu)$ can reach the points $b, a$ used in \eqref{eqn:def_stop_times_above}, \eqref{eqn:def_stop_times_below} in finite time, the analysis to show bounds as in Theorem \ref{thm:renewal} goes through for any initial conditions. Therefore, $\expec_{(h, \nu)} e^{-\zeta} > 0$ for all $(h, \nu)$ and is bounded below on bounded sets. Setting $\alpha = \inf_{(h, \nu) \in \Lambda } \sum_{n=0}^\infty \expec_{(h, \nu)} \> e^{-\zeta_n} $, \eqref{eqn:resolvent_structural} gives (c). Note that setting $A = S$ in \eqref{eqn:resolvent_structural} shows $\alpha = \left(1 -  \expec_{\renewalpt} \> e^{-\zeta} \right)^{-1}\>\inf_{(h, \nu) \in \Lambda } \expec_{(h, \nu)} \> e^{-\zeta}$.

Since, as noted in the previous paragraph, $\prob_{(h, \nu)}\left(\zeta < \infty \right) = 1$ for all $(h, \nu) \in S$ and $\renewalpt$ is contained in $\Lambda$, we have $\prob_{(h, \nu)}\left(\tau_\Lambda < \infty \right) = 1$. Moreover, part (c) of the lemma shows that the set $\Lambda$ is \emph{petite} in the sense of Section 4 in \cite{stability}. These two observations imply (d) by Theorem 4.3 (ii) of \cite{stability}.
\end{proof}

\begin{lemma} \label{lemma:lyapunov_exponential_moment} For $\Lambda$ as in Lemma \ref{lemma:structural}, there exists a continuous function $F$ such that,
\begin{eqnarray} \label{eqn:lyap_expo_moment}
    \expec_{(h, \nu)} \, e^{\eta \tau_\Lambda(1)} \leq F(h, \nu),  \quad \text{for all }\> \eta < \frac{1}{16}(g/(1+\gamma))^2, \> (h, \nu) \in S,
\end{eqnarray}
where $\tau_\Lambda(1) \coloneqq \left\{t \geq 1 \> : \> \left(H_t, V_t\right) \in \Lambda\right\}$. In particular, the function $\expec_{(h, \nu)} \, e^{\eta \tau_\Lambda(1)} $ is finite for all $(h, \nu)$ and bounded on $\Lambda$, uniformly for $\eta < \frac{1}{16}(g/(1+\gamma))^2$.
\end{lemma}

\begin{figure}
    \centering
    \includegraphics{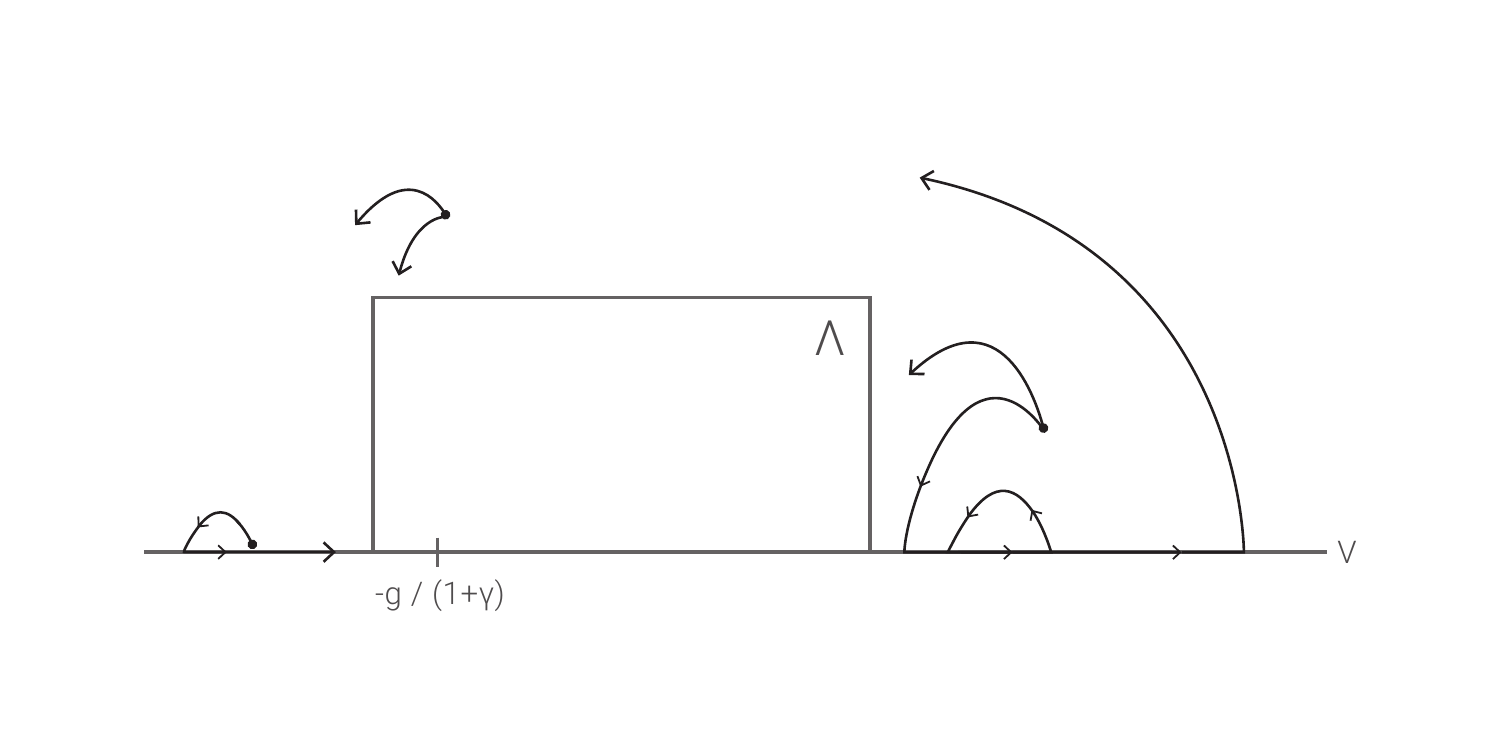}
    \caption{Process with initial conditions in left, top, right regions relative to $\Lambda$.}
    \label{fig:exponential_moment}
\end{figure}

\begin{proof} 
Throughout the proof, the terms $c, c' > 0$ will denote constants dependent only on $g, \gamma$, not $\eta$, whose values might change from line to line. We will define a range of $\eta$ to satisfy the theorem, giving upper bounds on such admissible $\eta$ as needs arise. 

First we bound $\expec_{(h, \nu)} \, e^{\eta \tau_\Lambda}$ for $(h, \nu) \not\in \Lambda$, recalling $\tau_\Lambda = \inf\{t \geq 0  \> : \> \left(H_t, V_t\right) \in \Lambda \}$. We consider three regions for initial conditions $(h, \nu)$: the left, the top and the right of the box $\Lambda$.

Define $\tilde{a} =  -(g + g/2(1+\gamma))/(1+\gamma)$. Consider initial conditions to the left of $\Lambda$, that is $(H_0, V_0) = (h, \nu)$ for $h \geq 0$ and $\nu \in (-g/\gamma, \tilde{a})$. Recall \eqref{eqn:increase_on_boundary_only} says $V$ increases only on the set $\{s: \, H_s = 0 \}$. This implies $\tau_{\tilde{a}}^V = \tau_\Lambda$. Lemma \ref{lemma:vel_increase_below_renewal_tailprob} with $\epsilon_0 = g/2(1+\gamma)$ and $u = \tilde{a}$ show,
\begin{equation} \label{eqn:lyapunov_from_left}
    \underset{\nu \in (-g/\gamma, \tilde{a})}{\sup}\expec_{(h, \nu)} \, e^{\eta \tau_\Lambda} \leq e^c e^{4 \eta (1+\gamma ) h / g}, \quad h \ge 0, \ \nu \in (-g/\gamma, \tilde{a}), \ \eta < \frac{1}{16}\left(g/(1+\gamma) \right)^2.
\end{equation}
Consider $V_0 = \tilde{a}$ and $H_0 > 1$. \eqref{eqn:increase_on_boundary_only} again shows $\tau_{\tilde{a}}^{V +} = \inf \{t > 0 \>:\> V_t = \tilde{a} \} = \tau_\Lambda $. Remark \ref{tauplus} after Lemma \ref{lemma:vel_increase_below_renewal_tailprob} gives,
\begin{equation} \label{eqn:lyapunov_from_left_plus}
    \expec_{(h, \tilde{a})} \, e^{\eta \tau_\Lambda} \leq e^c e^{4 \eta (1+\gamma ) h / g}, \quad h > 1, \> \eta < \frac{1}{16}\left(g/(1+\gamma) \right)^2.
\end{equation}
Now take $h > 1, \nu \in (\tilde{a}, \frac{g}{\gamma}]$, that is, initial conditions above $\Lambda$. Since $V$ cannot increase before $H$ hits zero, there are only two ways the process can enter $\Lambda$ as depicted in Figure \ref{fig:exponential_moment}: Through the top of the rectangle, when $\tau^H_1 \le \tau_{\tilde{a}}^V$, or if the velocity drops below $\tilde{a}$ and the process enters $\Lambda$ when $V$ next rises to hit $\tilde{a}$. 
\eqref{eqn:system_starts_interior} shows $\left\{\tau^H_1 \le \tau_{\tilde{a}}^V, \,\, \tau_1^H > \frac{1}{\gamma} \log \left(\frac{\nu + g/\gamma}{\tilde{a} + g/\gamma} \right)\right\} = \varnothing$. As a result, 
\begin{equation} \label{eqn:lyapunov_from_top0}
    \expec_{(h, \nu)} \> \left( e^{\eta \, \tau_\Lambda} \indi{\tau^H_1 \le \tau_{\tilde{a}}^V}  \right) = \expec_{(h, \nu)} \> \left( e^{\eta \, \tau^H_1} \indi{\tau^H_1 \le \tau_{\tilde{a}}^V}  \right) \leq \left(\frac{\nu + g/\gamma}{\tilde{a} + g/\gamma}\right)^{\eta/\gamma}\leq e^c,
\end{equation}
for $\tilde{a} < \nu \leq g/\gamma$ and $\eta < \frac{1}{16}\left(g/(1+\gamma) \right)^2$. In addition, for the same range on $\nu,\eta$, using the strong Markov property at $\tau_{\tilde{a}}^V$ and \eqref{eqn:lyapunov_from_left_plus},
\begin{equation} \label{eqn:lyapunov_from_top1}
    \expec_{(h, \nu)} \, \left( e^{\eta \, \tau_\Lambda} \indi{\tau_{\tilde{a}}^V < \tau^H_1} \right) \leq e^c \> \expec_{(h, \nu)} \, \left(\indi{\tau_{\tilde{a}}^V < \tau^H_1}  e^{(4 \eta (1+\gamma )/g) H_{\tau_{\tilde{a}}^V}}\right).
\end{equation}
Now \eqref{eqn:system_starts_interior} implies $H_t \leq h + c' + \sup_{t < \infty}\left(-B_t - t\frac{g}{\gamma}\right)$ for $t < \tau_{\tilde{a}}^V\wedge \tau_1^H$. Since $\sup_{t < \infty}\left(-B_t - t\frac{g}{\gamma}\right)$ is exponentially distributed with mean $\frac{\gamma}{2g}$, \eqref{eqn:lyapunov_from_top1} gives,
\begin{equation} \label{eqn:lyapunov_from_top2}
    \expec_{(h, \nu)} \, \left( e^{\eta \, \tau_\Lambda} \indi{\tau_{\tilde{a}}^V < \tau^H_1} \right) \leq e^c e^{4 \eta (1+\gamma ) h / g}.
\end{equation}
for any $h > 1$, $\tilde{a} < \nu \le g/\gamma$ and $\eta < \frac{1}{16}\left(g/(1+\gamma) \right)^2$. From \eqref{eqn:lyapunov_from_top0} and \eqref{eqn:lyapunov_from_top2}, we obtain
\begin{equation}\label{topbd}
\expec_{(h, \nu)} \, \left( e^{\eta \, \tau_\Lambda}\right) \le e^c e^{4 \eta (1+\gamma ) h / g}, \quad h > 1, \ \tilde{a} < \nu \le g/\gamma, \ \eta < \frac{1}{16}\left(g/(1+\gamma) \right)^2.
\end{equation}
Lastly we consider the region to the right of $\Lambda$, where $\nu > \frac{g}{\gamma}$. When the velocity hits $g/\gamma$, either it is in $\Lambda$ or $H > 1$ and the process has `jumped over' the box's right edge. See Figure \ref{fig:exponential_moment}. Lemma \ref{lemma:velocity_tail_bounded_by_exponential} with $u = g/\gamma$ shows
\begin{equation} \label{eqn:lyapunov_from_right1}
    \expec_{(h, \nu)} \, \left(e^{\eta \tau_\Lambda} \indi{H_{\tau_{g/\gamma}^V} \leq 1} \right) \leq \expec_{(h, \nu)} \, e^{\eta \, \tau_{g/\gamma}^V} \leq e^c \, e^{2\frac{g}{\gamma}\nu},
\end{equation}
for $\eta < 2g^2 / \gamma$. In addition, the strong Markov property at $\tau_{g/\gamma}^V$, along with \eqref{topbd}, \eqref{eqn:lyapunov_from_right1} and the Cauchy-Schwarz inequality imply,
\begin{multline} \label{eqn:lyapunov_from_right2}
    \expec_{(h, \nu)} \, \left(e^{\eta \tau_\Lambda} \indi{H_{\tau_{g/\gamma}^V} > 1} \right) \leq e^c \> \expec_{(h, \nu)} \, \left(e^{\eta \tau_{g/\gamma}^V + \left(4 \eta (1+\gamma )/ g\right)H_{\tau_{g/\gamma}^V}} \right)\\
    \le e^c\sqrt{\expec_{(h, \nu)} \, \left(e^{2\eta \tau_{g/\gamma}^V}\right)}\sqrt{\expec_{(h, \nu)} \, e^{\left(8 \eta (1+\gamma )/ g\right) H_{\tau_{g/\gamma}^V}}}
    \leq e^c \> e^{2\nu \frac{g}{\gamma}} \>  \sqrt{\expec_{(h, \nu)} \, e^{\left(8 \eta (1+\gamma )/ g\right) H_{\tau_{g/\gamma}^V}}},
\end{multline}
for $\eta < g^2/\gamma$ sufficiently small that the second term exists. We now define the allowable range of $\eta$ and bound the second term in \eqref{eqn:lyapunov_from_right2}. Using \eqref{eqn:system}, we have for all $(h, \nu) \in [0, \infty) \times (g/\gamma, \infty)$ and $t < \tau_{g/\gamma}^V$,
\begin{equation}\label{eqn:lyapunov_from_right3}
    H_t + V_t / \gamma = h + \nu /\gamma - B_t - tg/\gamma + (1+1/\gamma)L_t \leq h + \nu /\gamma + \sup_{t < \infty} \left( - B_t - tg/\gamma \right) + (1+1/\gamma)\sup_{t < \infty}\left(B_t - tg/\gamma \right).
\end{equation}
If $\eta$ is small enough that $(1+1/\gamma ) \times 16 \eta (1+\gamma )/ g < g/\gamma$, in other words $\eta < \frac{1}{16}(g/(1+\gamma))^2$, then since $E_1^* := \sup_{t < \infty}\left(-B_t - tg/\gamma \right)$ and $E_2^* := \sup_{t < \infty}\left(B_t - tg/\gamma \right)$ are exponentially distributed with mean $\gamma / 2g$, \eqref{eqn:lyapunov_from_right3} gives 
\begin{multline} \label{eqn:lyapunov_from_right4}
    \expec_{(h, \nu)} \, e^{\left(8 \eta (1+\gamma )/ g\right) \left(\gamma^{-1}V_{t \wedge \tau_{g/\gamma}^V} +  H_{t \wedge \tau_{g/\gamma}^V}\right)} \le e^{8\eta(1+\gamma)(\gamma h+\nu)/g\gamma}\expec \> e^{(8\eta(1+\gamma)/g)(E_1^* + (1+1/\gamma)E_2^*)}\\
    \le e^{8\eta(1+\gamma)(\gamma h+\nu)/g\gamma} \sqrt{\expec\> e^{(16\eta(1+\gamma)/g)E_1^*}}\sqrt{\expec\> e^{(16\eta(1+\gamma)(1+1/\gamma)/g)E_2^*}} \leq e^c e^{8\eta(1+\gamma)(\gamma h+\nu)/g\gamma},
\end{multline}
for any $t > 0$, where the third bound above follows from Cauchy-Schwarz inequality. By Fatou's lemma,
\begin{equation} \label{eqn:lyapunov_from_right5}
    \expec_{(h, \nu)} \, e^{ \left(8 \eta (1+\gamma )/ g\right) H_{\tau_{g/\gamma}^V}} \leq e^c e^{8\eta(1+\gamma)(\gamma h+\nu)/g\gamma}.
\end{equation}
Combining \eqref{eqn:lyapunov_from_right5} and \eqref{eqn:lyapunov_from_right2},
\begin{equation} \label{eqn:lyapunov_from_right6}
    \expec_{(h, \nu)} \, \left(e^{\eta \tau_\Lambda} \indi{H_{\tau_{g/\gamma}^V} > 1} \right) \leq e^c \> e^{2\nu \frac{g}{\gamma}} \>  e^{4\eta(1+\gamma)(\gamma h+\nu)/g\gamma},
\end{equation}
for $h \geq 0, \nu > g/\gamma$ and $\eta < g^2/\gamma \wedge \frac{1}{16}(g/(1+\gamma))^2 = \frac{1}{16}(g/(1+\gamma))^2$. From \eqref{eqn:lyapunov_from_right1} and \eqref{eqn:lyapunov_from_right6}, we obtain
\begin{equation}\label{rightbd}
\expec_{(h, \nu)} \, \left(e^{\eta \tau_\Lambda}\right) \leq e^c \> e^{2\nu \frac{g}{\gamma}} \>  e^{4\eta(1+\gamma)(\gamma h+\nu)/g\gamma}, \quad h \geq 0, \ \nu > g/\gamma, \ \eta < \frac{1}{16}(g/(1+\gamma))^2.
\end{equation}
Summarizing \eqref{eqn:lyapunov_from_left}, \eqref{eqn:lyapunov_from_left_plus}, \eqref{topbd} and \eqref{rightbd},
\begin{equation} \label{eqn:lyapunov_bound_semifinal}
   \expec_{(h, \nu)} \, e^{\eta \tau_\Lambda} \leq e^c \> e^{2\nu \frac{g}{\gamma} + 4\eta(1+\gamma)(\gamma h+\nu)/g\gamma}, \quad (h, \nu) \not \in \Lambda, \ \eta < \frac{1}{16}(g/(1+\gamma))^2.
\end{equation}
Notice that to introduce the $2\nu g/\gamma$ term in \eqref{eqn:lyapunov_bound_semifinal} when applying \eqref{eqn:lyapunov_from_left}, \eqref{eqn:lyapunov_from_left_plus}, \eqref{topbd}, we need only use the fact that $1 = e^{-2\nu g/\gamma} e^{2\nu g/\gamma} < e^{2(g/\gamma)^2} e^{2\nu g/\gamma}$, since $\nu > -g/\gamma$.  

We complete the proof of the theorem. By definition, $\tau_\Lambda(1) = 1 + \inf\{t \geq 0 \> | \> (H_{1 + t}, V_{1+t}) \in \Lambda \}$, so for any $(h, \nu)$ we have by \eqref{eqn:lyapunov_bound_semifinal},
\begin{multline} \label{eqn:lyapunov_bound_semifinal1}
   \expec_{(h, \nu)} \, e^{\eta \tau_\Lambda(1)} =  \expec_{(h, \nu)} \, \left(\indi{(H_{1}, V_{1}) \in \Lambda} \, e^{\eta \tau_\Lambda(1)}\right) + \expec_{(h, \nu)} \, \left(\indi{(H_{1}, V_{1}) \not \in \Lambda} \, e^{\eta \tau_\Lambda(1)}\right)\\
   \le e^\eta + e^\eta \expec_{(h, \nu)} \, \left(\indi{(H_{1}, V_{1}) \not \in \Lambda} \expec_{(H_1, V_1)} \,e^{\eta \tau_\Lambda}\right)
   \leq e^c \> \left( 1 + \expec_{(h, \nu)} \, e^{2V_1 \frac{g}{\gamma} + 4\eta(1+\gamma)(\gamma H_1 + V_1)/g\gamma}\right),
\end{multline}
for $\eta$ sufficiently small. We use the crude bounds $V_t \leq \nu + \sup_{s \leq 1}\left(B_s + s\frac{g}{\gamma}\right)$, $L_t \le \sup_{s \leq 1}\left(B_s + s\frac{g}{\gamma}\right)$ for $t \leq 1$ and so $H_1 \leq h + \int_0^1 V_t \, dt - B_1 + \sup_{s \leq 1}\left(B_s + s\frac{g}{\gamma}\right) \leq h + \nu - B_1 + 2\sup_{s \leq 1}\left(B_s + s\frac{g}{\gamma}\right)$. From this, \eqref{eqn:lyapunov_bound_semifinal1} and the finiteness of exponential moments of $-B_1$ and $\sup_{s\leq 1}\left(B_s + s\frac{g}{\gamma} \right)$ we calculate,
\begin{multline}\label{eqn:lyapunov_final}
    \expec_{(h, \nu)} \, e^{\eta \tau_\Lambda(1)} \leq e^c \, \left(1 + e^{2\nu \frac{g}{\gamma} + 4\eta(1+\gamma)(\gamma h+ (1+\gamma)\nu)/g\gamma}  \right)\leq e^c \, \left(1 + e^{9 \nu g/4\gamma  + h g/4(1+\gamma)}  \right) \coloneqq F(h, \nu),
\end{multline}
for $\eta < \frac{1}{16}(g/(1+\gamma))^2$ and any $(h,\nu) \in S$.
\end{proof}
The last ingredient needed in the proof of exponential ergodicity is a \emph{minorization condition} (e.g. Assumption 2 of \cite{hairandmat}, or \cite{ergodicity} for a different formulation): There exist $t_0, \alpha > 0$ and a non-trivial measure $\mu$ such that
\begin{equation}\label{eqn:minorization_cond}
    P^{t_0}((h, \nu), \> \cdot\>)) \geq \alpha \mu(\cdot),  \quad \quad (h, \nu) \in \Lambda.
\end{equation}
In essence, \eqref{eqn:minorization_cond} ensures two versions of the process started within $\Lambda$ can be coupled. 
Typically, \eqref{eqn:minorization_cond} is checked by showing $P^t$ has a continuous density with respect to Lebesgue measure for each $t>0$. As stated before, these techniques are not available in our setup. Instead Lemmas \ref{lemma:structural} and \ref{thm:ergodic}, which use the renewal structure of our process in a crucial way, establish the following stronger version of \eqref{eqn:minorization_cond}. 
\begin{lemma} \label{lemma:minorization} Fix $\Lambda$ as in Lemma \ref{lemma:structural}. There exists a $t_0 > 0$ a non-trivial measure $\tilde{\mu}$ such that for all measurable $A \subset S$, $$P^t((h, \nu),\> A) \geq \tilde{\mu}(A) \quad \quad (h, \nu) \in \Lambda ,  \> t \geq t_0.$$
\end{lemma}
\begin{proof}
Theorem \ref{thm:ergodic} shows that the process $(H_t,V_t)_{t\ge 0}$ is ergodic which, via Theorem 6.1 of \cite{stability}, implies that there is a skeleton chain that is irreducible. Moreover, from Lemma \ref{lemma:structural} (c), $\Lambda$ is `petite.' The Lemma now follows from Proposition 6.1 of \cite{stability} upon noting that our process in positive Harris recurrent, which was proved in Lemma \ref{lemma:structural} (d).
\end{proof}

\begin{proof}[Proof of Theorem \ref{thm:expo_ergodic}]
The theorem follows from Theorems 5.2, 6.2 of \cite{ergodicity}, which extend the discrete time drift and minorization conditions of Harris's theorem to continuous-time processes and show how these can be used to obtain exponential ergodicity (see also \cite{hairandmat} for a discrete-time formulation). Specifically, Lemma \ref{lemma:structural} (c) shows that $\Lambda$ is `petite'. Theorem 6.2 of \cite{ergodicity} (with $f \equiv 1$ and $\delta=1$) and Lemma \ref{lemma:lyapunov_exponential_moment} show the function $G(h, \nu) = 1 + \frac{1}{\eta}\left(\expec_{(h, \nu)} \, e^{\eta \tau_\Lambda(1)} - 1\right)$ is a Lyapunov function for the process which satisfies the following \emph{drift condition} for every $t_0 > 0$, 
\begin{equation}\label{eqn:drift_cond}
    P^tG \leq \lambda_{t_0} G + c\indi{\Lambda} \quad t \leq t_0, \> c > 0, \lambda_{t_0} < 1.
\end{equation}
The result then follows from Theorem 5.2 of \cite{ergodicity}. We note here that the referenced theorem requires an `aperiodicity' type condition for the semigroup $P^t$, given on p. 1675 of the reference (which is not the same as the more standard notion of aperiodicity defined in \eqref{eqn:aperiodicity_skeleton} below). However, the first line in the proof of Theorem 5.2 of \cite{ergodicity} makes clear that the sole use of the `aperiodicity' condition is to ensure there exists a $t_0 > 0$ such that the discrete-time process with transitions $\{P^{kt_0} \}_{k\ge 1}$ is geometrically ergodic. In other words,
\begin{equation}\label{eqn:skeleton_ergodic}
    \|P^{kt_0}((h, \nu), \> \cdot) - \pi \|_{TV} \le G(h, \nu) D_0 r^k, \quad k \ge 1, \> (h, \nu) \in S,
\end{equation}
for some $t_0 > 0$ and constants $D_0 \in (0, \infty)$ and $r \in (0, 1)$,  with $G$ as in \eqref{eqn:drift_cond}. 

For clarity, we verify here that \eqref{eqn:skeleton_ergodic} holds, from which the proof of Theorem 5.2 in \cite{ergodicity} can proceed as written. First, we recall the definition of aperiodicity for a discrete-time Markov chain (see Section 5.4.3 of \cite{meyntweedie}, Theorem 5.4.4 and the discussion preceding it): A process with transitions $\{Q^{k} \}_{k\ge 1}$ is called \emph{aperiodic} if there exists a set $C$ such that \eqref{eqn:minorization_cond} holds for $Q^k$ in place of $P^{t_0}$ for some $k \geq 1$ and $C$ in place of $\Lambda$, and such that, 
\begin{equation}\label{eqn:aperiodicity_skeleton}
    gcd \{n \geq 1 \> : \>  Q^n((h, \nu), \cdot\>) \geq \alpha_n \mu(\cdot), \quad (h, \nu) \in C \quad \text{some } \alpha_n > 0 \} = 1.
\end{equation}
Lemma \ref{lemma:minorization} shows \eqref{eqn:aperiodicity_skeleton} is satisfied for the chain with transition laws $Q^k = P^{kt_0}$, $C = \Lambda$ and $\alpha_n = 1$ for all $n$, where $\Lambda$ and $t_0$ are as given in the Lemma. Thus, the discretely sampled chain with transition laws $\{P^{kt_0} \}_{k\ge 1}$ is aperiodic. \eqref{eqn:skeleton_ergodic} follows immediately from this observation in conjunction with \eqref{eqn:drift_cond} and Theorem 16.0.1 (ii), (iv) of \cite{meyntweedie}. The proof of Theorem 5.2 in \cite{ergodicity} now proceeds exactly as written, giving the result.
\end{proof}
\begin{appendix}
\section{Appendix: Hitting time estimates}\label{hittimeest}
The following lemmas primarily support statements leading to the proof of Theorem \ref{thm:renewal}. We begin with the observation that $\tau_0^H \coloneqq \sigma(0)$ is finite a.s. for all initial conditions.
\begin{lemma} \label{lemma:mgf_hitting_boundary} For each $(h, \nu) \in S$, $$\prob_{(h, \nu)}\left(\sigma(0) < \infty \right) = 1$$
\end{lemma}

\begin{proof} 
$\sigma(0) = 0$ for $(h, \nu) \in \partial S = \{0\} \times (-g/\gamma, \infty)$. For $(h, \nu) \in S^\circ$ and $t < \sigma(0)$, recall by \eqref{eqn:system_starts_interior},
$$
H_t \le h + \frac{\nu}{\gamma} + \frac{g}{\gamma^2} - B_t - \frac{g}{\gamma}t
$$
The lemma follows by noting that the bounding process in the above inequality is a Brownian motion with negative drift starting from a positive point and hence hits zero in finite time almost surely.
\end{proof}
\begin{lemma} \label{lemma:vel_escape_interval}
For any  $-g/\gamma < a < b$, $\ell > \frac{b - a}{\gamma a + g} > 0$ and $m\geq 1$,
$$\underset{\nu \in [a, b]}{\sup}\mathsmaller{\prob_{(0, \nu)}\left(\tau^V_{[a, b]^c} > m(\ell+1) \right) \leq \left[\prob\left(B_1 \leq b-a + (1+\gamma)b + g \right)\right]^m }.$$
\end{lemma}

\begin{proof}
For $t \le \tau^V_{[a, b]^c}$ and $\nu \in [a, b]$, recalling that $L_t = \sup_{u\leq t}\left(B_u-S_u\right)$ and from \eqref{eqn:system}, we have $L_t \geq B_t - S_t \geq B_t - tb$ and $b \geq V_t \geq \nu - (\gamma b + g)t + L_t$. This gives 
$$B_t \le b - \nu + ((1+\gamma) b +g)t \le b - a + ((1+\gamma) b +g)t$$
for all $t \leq \tau^V_{[a, b]^c}$. Fix $\ell > \frac{b-a}{\gamma a + g} > 0$. Suppose $\tau^V_{[a, b]^c} \ge 1 + \ell$. Then if $\sigma(1) > 1+ \ell$, $L_t$ must be constant on $[1, 1+\ell]$, and using \eqref{eqn:system} and \eqref{eqn:increase_on_boundary_only} shows $$V_{1 + \ell} - V_1 \leq -\ell(\gamma a + g) < a - b \quad \quad \imply \quad V_{1+\ell} < a ,$$ a contradiction. Thus, $\tau^V_{[a, b]^c} \ge 1 + \ell$ implies $\sigma(1) \le 1+ \ell$. For any $m \ge 1$, applying the strong Markov property at $\sigma(1)$,
\begin{multline*}
\sup_{\nu \in [a, b]}\prob_{(0, \nu)}\left(\tau^V_{[a, b]^c} > m(\ell+1) \right)\\
\le \sup_{\nu \in [a, b]}\mathbb{E}_{(0, \nu)}\left(\mathbbm{1}_{\{B_1 \leq b-a + (1+\gamma)b + g , \ \sigma(1) \le 1+ \ell, \ V_{\sigma(1)} \in [a,b]\}}\prob_{(0, V_{\sigma(1)})}\left(\tau^V_{[a, b]^c} > m(\ell+1) - (\ell + 1)\right)\right)\\
\le \prob\left(B_1 \leq b-a + (1+\gamma)b + g \right) \sup_{\nu \in [a, b]}\prob_{(0, \nu)}\left(\tau^V_{[a, b]^c} > (m-1)(\ell+1)\right).
\end{multline*} 
An induction argument gives the result. 
\end{proof}
\begin{lemma}\label{lemma:gap_hits_zero_before_long} Fix $b = -\frac{g - g/2(1+\gamma)}{1+ \gamma}$ and $a =-\frac{g + g/2\gamma}{1+ \gamma} $. There exist positive constants $c$ and $t'(\gamma, g)$ such that
\begin{equation*}
    \underset{(h, \nu) \in [0, tg/4\gamma(1+\gamma)]\times(a, b)}{\sup}\>\prob_{(h, \nu)}\left(\sigma(\tau_a^V) - \tau_a^V > t ,\> \tau_a^V < \tau_b^V\right) \leq e^{-ct},
\end{equation*}
for all $t > t'(\gamma, g)$.
\end{lemma}
\begin{proof}
Lemma \ref{lemma:vel_escape_interval} shows there exists a $\tilde{p} \in (0, 1)$ and $C > 0$ depending on $\gamma, g$ such that for $t\geq C$,
\begin{equation}\label{eqn:kzero_1}
    \underset{\nu \in (a, b)}{\sup}\>\prob_{(0, \nu)}\left(t < \tau_a^V \wedge \tau_b^V\right) \leq \tilde{p}^{t/C}.
\end{equation}
Suppose $h > 0$. \eqref{eqn:increase_on_boundary_only} shows $V$ cannot reach the level $b > \nu$ without the gap process $H$ hitting zero, which implies $\sigma(0) < \tau_b^V$. In addition, \eqref{eqn:hitting_time_velocity} implies $\tau_a^V \geq \frac{1}{\gamma}\log\left(\frac{\nu + g/\gamma}{a + g/\gamma} \right)$ with equality when $\sigma(0) \ge \tau_a^V$. Therefore if $\sigma(0) \geq \tau_a^V$,  $\tau_a^V\wedge \tau_b^V \leq \frac{1}{\gamma}\log\left(\frac{\nu + g/\gamma}{a + g/\gamma} \right) \leq \frac{1}{\gamma}\log\left(\frac{b + g/\gamma}{a + g/\gamma} \right)$. Hence, for $t > \frac{1}{\gamma}\log\left(\frac{b + g/\gamma}{a + g/\gamma} \right)$, we obtain
\begin{multline}\label{eqn:kzero_2}
    \prob_{(h, \nu)}\left(t < \tau_a^V\wedge \tau_b^V \right) = \prob_{(h, \nu)}\left(\sigma(0) < \tau_a^V, \> t < \tau_a^V\wedge \tau_b^V \right) \\
    = \prob_{(h, \nu)}\left(\frac{1}{\gamma}\log\left(\frac{\nu + g/\gamma}{a + g/\gamma} \right) < \sigma(0) < \tau_a^V, \> t < \tau_a^V\wedge \tau_b^V \right) + \prob_{(h, \nu)}\left(\sigma(0) \leq \tau_a^V \wedge \frac{1}{\gamma}\log\left(\frac{\nu + g/\gamma}{a + g/\gamma} \right), \> t < \tau_a^V\wedge \tau_b^V \right)\\
    = \prob_{(h, \nu)}\left(\sigma(0) \leq \frac{1}{\gamma}\log\left(\frac{\nu + g/\gamma}{a + g/\gamma} \right), \> t < \tau_a^V\wedge \tau_b^V \right).
\end{multline}
\eqref{eqn:increase_on_boundary_only} shows that when $\sigma(0) < \tau_a^V$, $V_{\sigma(0)} \in (a, b)$. As a result, \eqref{eqn:kzero_1}, \eqref{eqn:kzero_2} and the strong Markov property at $\sigma(0)$ show there exists positive constants $c, C$ and $t_0(\gamma, g)$ such that that for any $(h, \nu) \in (0, \infty) \times (a, b)$ and $t > t_0(\gamma, g)$,
\begin{multline}\label{eqn:kzero_3}
    \prob_{(h, \nu)}\left(t < \tau_a^V\wedge \tau_b^V \right) = \expec_{(h, \nu)}\left(\prob_{(0, V_{\sigma(0)})} \left(t - \sigma(0) < \tau_a^V\wedge \tau_b^V \right)\mathbbm{1}_{\sigma(0) \leq \frac{1}{\gamma}\log\left(\frac{\nu + g/\gamma}{a + g/\gamma} \right)}\right)\\
    \leq \underset{\nu \in (a, b)}{\sup}\>\prob_{(0, \nu)}\left(t - \frac{1}{\gamma}\log\left(\frac{b + g/\gamma}{a + g/\gamma}\right) < \tau_a^V \wedge \tau_b^V\right) \leq e^C \> e^{-ct} \leq e^{-ct/2}.
\end{multline}
Starting from $(H_0,V_0)=(0,\nu)$ for some $\nu \in (a,b)$, $\sigma(\tau_a^V) - \tau_a^V$ is large if one of the following two events happen: (i) either the gap is large at $\tau_a^V$ or (ii) the gap at $\tau_a^V$ is not large but the gap remains positive for a large time after $\tau_a^V$. We handle these cases separately and show that in either case, the Brownian particle has to attain a large negative value before $\tau_b^V$, leading to exponentially small probability bounds.

To estimate the probability of the event (i), fix $t>t_0(\gamma, g)$ and set $\left(H_0, V_0 \right) = (h, \nu) \in [0, tg/4\gamma(1+\gamma)]\times(a, b)$. System equations \eqref{eqn:system} show $H_u - V_u = h-\nu + (1+\gamma)S_u + ug - B_u$. For $u < \tau_a^V \wedge \tau_b^V \wedge t$,
\begin{multline}\label{eqn:kzero_4}
    H_u - tg/2\gamma = H_u - V_u + V_u - tg/2\gamma = h - \nu + (1+\gamma)S_u + ug - B_u + V_u - tg/2\gamma  \\ 
    \leq tg/4\gamma(1+\gamma) + b - a + ug/2(1+\gamma) - tg/2\gamma + \sup_{u \leq t}\left(- B_u\right)\\
    \leq tg/4\gamma(1+\gamma) + b - a + t\left(g/2(1+\gamma) - g/2\gamma\right) + \sup_{u \leq t}\left(- B_u\right)
    = b-a - t g/4\gamma(1+\gamma) + \sup_{u \leq t}\left(- B_u\right),
\end{multline}
where the first inequality follows from $S_u \leq ub = -u\left(\frac{g-g/2(1+\gamma)}{1+\gamma}\right)$. If $\tau_a^V < \tau_b^V$ and $\tau_a^V \wedge \tau_b^V \leq t $, \eqref{eqn:kzero_4} shows that if $H_{\tau_a^V} > tg/2\gamma + b - a$ then $\sup_{q \leq t}\left(- B_q\right) > t g/4\gamma(1+\gamma)$. Using \eqref{eqn:kzero_4}, a standard upper bound on the Gaussian distribution and the reflection principle, we obtain
\begin{multline}\label{eqn:kzero_5}
    \underset{(h, \nu) \in [0, tg/4\gamma(1+\gamma)]\times(a, b)}{\sup}\>\prob_{(h, \nu)}\left(\tau_a^V < \tau_b^V, \> \tau_a^V \wedge \tau_b^V \leq t, \>  H_{\tau_a^V} > tg/2\gamma + b-a \right)\\
    \leq \prob\left(\underset{u \leq t}{\sup}\left(- B_u\right) > t g/4\gamma(1+\gamma)\right)
    \leq \frac{8(1+\gamma)\gamma}{g\sqrt{2\pi t}} e^{-tg^2/32\gamma^2(1+\gamma)^2}.
\end{multline}
To estimate the probability of event (ii), fix $t > t_0(\gamma, g)$ and set $(H_0, V_0) = (h , a)$ with any $h \in \left(0, tg/2\gamma + b-a \right]$. When $t < \sigma(0)$, \eqref{eqn:system_starts_interior} implies $0 < H_t \leq tg/2\gamma + b-a + a/\gamma + g/\gamma^2 - B_t - tg/\gamma = c' - B_t - tg/2\gamma$ for a positive constant $c'$. We choose $t_0(\gamma, g)$ large enough that $t_0(\gamma, g)g/2\gamma - c' > 0$. Therefore,
\begin{equation}\label{eqn:kzero_6}
    \underset{h \in (0, tg/2\gamma + b-a]}{\sup} \> \prob_{(h, a)}\left(t < \sigma(0) \right) \leq \prob\left(-B_t >tg/2\gamma - c'  \right) \leq \frac{\sqrt{t}}{\sqrt{2\pi}\left(tg/2\gamma - c' \right)} e^{-(tg/2\gamma - c')^2/2t}.
\end{equation}
In \eqref{eqn:kzero_3}, \eqref{eqn:kzero_5} and \eqref{eqn:kzero_6} the choice of $t > t_0(\gamma, g)$ was arbitrary. As a result, \eqref{eqn:kzero_3} and \eqref{eqn:kzero_5} show that for a positive constant $c$,
\begin{multline}\label{eqn:kzero_7}
    \prob_{(h, \nu)}\left(\sigma(\tau_a^V) - \tau_a^V > t, \> \tau_a^V < \tau_b^V\right)\\
    \leq \prob_{(h, \nu)}\left(t < \tau_a^V \wedge \tau_b^V\right) + \prob_{(h, \nu)}\left(\tau_a^V < \tau_b^V, \> \tau_a^V \wedge \tau_b^V \leq t, \>  H_{\tau_a^V} > tg/2\gamma + b-a \right) \\
    + \prob_{(h, \nu)}\left(\sigma(\tau_a^V) - \tau_a^V > t, \>  H_{\tau_a^V} \leq tg/2\gamma + b-a , \>\tau_a^V < \tau_b^V, \> \tau_a^V \wedge \tau_b^V \leq t\right) \\
    \leq e^{-ct} + \prob_{(h, \nu)}\left(\sigma(\tau_a^V) - \tau_a^V > t, \>  H_{\tau_a^V} \leq tg/2\gamma + b-a \right),
\end{multline}
holds for all $(h, \nu) \in [0, tg/4\gamma(1+\gamma)]\times(a, b)$ and $t$ sufficiently large. The strong Markov property at $\tau_a^V$ and \eqref{eqn:kzero_6} show
\begin{equation}\label{eqn:kzero_8}
    \underset{(h, \nu) \in [0, tg/4\gamma(1+\gamma)]\times(a, b)}{\sup}\>\prob_{(h, \nu)}\left(\sigma(\tau_a^V) - \tau_a^V > t, \>  H_{\tau_a^V} \leq tg/2\gamma + b-a \right) \leq e^{-ct}.
\end{equation}
\eqref{eqn:kzero_7} and \eqref{eqn:kzero_8} prove the lemma.
\end{proof}

\begin{lemma} \label{lemma:vel_drop_above_tailbound} For any $u \in \left(-\frac{g}{1+\gamma}, 0\right)$, there exists a constant $t_0(u, \gamma, g) >0$ such that $$\underset{\nu \in \left(u, 0\right]}{\sup}\>\prob_{(0, \nu)}\left(\tau_u^V > t \right) \leq  \exp\left\{-\frac{\left( t\left((1+\gamma) u + g\right) + u - \nu \right)^2}{2t} \right\},$$
for $t > t_0(u, \gamma, g).$
\end{lemma}
\begin{proof}
For $t < \tau^V_u$,  $L_t \leq \underset{s \leq t}{\sup}(B_s - su) \leq \underset{s \leq t}{\sup}\,B_s - tu $ and

$$\underset{s \leq t}{\sup}\,B_s - t(1 + \gamma) u - tg + \nu \geq L_t - \int_0^t (\gamma V_s + g)\, ds + \nu = V_t > u .$$

The result follows by tail bounds on the supremum of Brownian motion, taking $t$ large enough that $0 < \frac{\sqrt{t}}{t\left((1+\gamma) u + g\right) + u} <  1$.
\end{proof}
The next lemma shows that the velocity process starting from a value $\nu \in (-\frac{g}{\gamma}, -\frac{g}{1+\gamma})$ cannot take too long to reach a larger value in the same interval because it is bounded below by a Brownian motion with positive drift.
\begin{lemma} \label{lemma:vel_increase_below_renewal_tailprob}
Fix $\epsilon_0 \in (0, g/ \gamma)$ and $u = - \frac{g + \epsilon_0}{1+\gamma}$. For each $h \geq 0$, there exists a constant $c >0$ such that 
$$\underset{\nu \in (-\frac{g}{\gamma}, u)}{\sup} \, \prob_{(h, \nu)} \left(\tau_u^V > t \right) \leq e^c \, e^{\epsilon_0h}\, e^{-\frac{\epsilon_0^2 }{2} t }, $$
for all $t > t_0 \coloneqq 2h/\epsilon_0 + 2g/\epsilon_0 \gamma (1+\gamma) + 4/\epsilon_0^2$. In particular, there exists a constant $c' > 0$ depending on $g, \gamma, \epsilon_0$ such that for any $\eta \in (0,  \epsilon_0^2/4)$,
$$\underset{\nu \in (-\frac{g}{\gamma}, u)}{\sup} \, \expec_{(h, \nu)} \> e^{\eta \tau_u^V} \leq e^{c'} e^{ 2 \eta h / \epsilon_0}.$$
\end{lemma}
\begin{proof}
For $(H_0, V_0) = (h,\nu)$ with $h \ge 0, \nu \in \left(-\frac{g}{\gamma}, u\right)$, and any $t < \tau_u^V$, 
$$
V_t \geq -t(\gamma u + g) + \underset{s\leq t}{\sup}(- h + B_s - su) + \nu \ge B_t -t((1+\gamma) u + g)  + \nu - h \ge B_t -t((1+\gamma) u + g)  -\frac{g}{\gamma} - h.
$$
This lower bound on $V$ implies
$$
\left\{\tau_u^V > t, \,\, (H_0, V_0) = (h,\nu)\right\}  \subset \left\{\inf \left\{s \geq 0 \, : \,\, B_s -s((1+\gamma)u + g) = u + \frac{g}{\gamma} + h\right\} > t\right\}.
$$
From tail bounds on hitting times of Brownian motion with drift $-((1+\gamma)u + g) = \epsilon_0$, we conclude that
\begin{multline*}
    \prob_{(h, \nu)} \left(\tau_u^V > t \right) \leq \int_t^\infty \frac{h + u + \frac{g}{\gamma}}{\sqrt{2\pi \, s^3}} \, \exp \left\{-\frac{\left(h + u + \frac{g}{\gamma} - \epsilon_0 s \right)^2}{2s} \right\} \, ds\\
  \leq \left(\frac{\epsilon_0}{2\theta} t +  \frac{1}{2} \right)^{-1}\, \frac{1}{\sqrt{2\pi}} \, \int_{z(t)}^\infty  \exp \left\{-\frac{z^2}{2} \right\} \, dz
    \leq \left(\frac{\epsilon_0}{\theta} t +  1 \right)^{-1}\, \frac{1}{z(t)} \, \exp \left\{-\frac{z(t)^2}{2} \right\} \leq e^{\epsilon_0 \theta} \, \exp \left\{-\frac{\epsilon_0^2 }{2} t \right\},
\end{multline*}
holds for all $\nu\in ( -g / \gamma, u)$ and for $t \geq 2\left(\frac{\theta}{\epsilon_0} \vee \frac{2}{\epsilon_0^2}\right)$, where $\theta = h + g / \gamma +u > 0$. The second inequality follows after a change of variables to $z(s) = \epsilon_0 \sqrt{s} - \frac{\theta}{\sqrt{s}}$. $t_0$ is chosen to satisfy $t_0 \ge 2\left(\frac{\theta}{\epsilon_0} \vee \frac{2}{\epsilon_0^2}\right)$ and $z(t) \ge \epsilon_0 \sqrt{t} - \frac{h + g/\gamma - g/(1+\gamma)}{\sqrt{t}} = \epsilon_0 \sqrt{t} - \frac{h + g/\gamma(1+\gamma)}{\sqrt{t}}  \ge 1$ for $t \geq t_0$. 

The second statement follows from $
\underset{\nu \in (-\frac{g}{\gamma}, u)}{\sup}\>\expec_{(h, \nu)}\, e^{\eta \, \tau_u^V} \leq \int_0^\infty \underset{\nu \in (-\frac{g}{\gamma}, u)}{\sup}\>\prob_{(h, \nu)} \left( e^{\eta \, \tau_u^V} > t \right) \, dt.
$
\end{proof}
\begin{remark}\label{tauplus}
Writing $\tau_u^{V+} = \inf\{t>0: V_t = u\}$, the proof of Lemma \ref{lemma:vel_increase_below_renewal_tailprob} shows that the same bounds obtained above hold for   $\prob_{(h, u)} \left(\tau_u^{V+} > t \right)$ and $\expec_{(h, u)} \> e^{\eta \tau_u^{V+}}$ when $h>0$.
\end{remark}
\begin{lemma} \label{lemma:velocity_tail_bounded_by_exponential}
Fix $u > 0$. For each $\nu > u$, 
$$\underset{h \geq 0}{\sup}\,\,\prob_{(h, \nu)}\left(\tau_u^V > t \right) \leq 1 \wedge \exp\left\{-2u\left(u - \nu + t(\gamma u + g)\right)\right\}, $$ 
for all $t \geq 0$. Therefore, if $\eta \in \left(0, u(\gamma u + g) \right)$
$$
\underset{h \geq 0}{\sup}\,\,\expec_{(h, \nu)}\, e^{\eta \, \tau_u^V} \leq e^{\frac{\eta}{\gamma u + g}\left(\nu - u \right)} + e^{2u\left(\nu - u \right)- \frac{\eta}{\gamma u + g}\left(\nu - u \right)} \leq 2 e^{2u\left(\nu - u \right)}.
$$
\end{lemma}

\begin{proof} Fix $(H_0, V_0) = (h, \nu) \in \Real_+ \times (u, \infty)$. The definition of $V$ gives for $t < \tau_u^V$,
\begin{equation*}
    u < V_t \leq  \nu - t(\gamma u + g) + 0\vee \underset{s \leq t}{\sup}\left(-h + B_s - su \right) \leq \nu - t(\gamma u + g) +  \underset{s < \infty}{\sup}\left(B_s - su \right). 
\end{equation*}
This bound implies the first claim of the lemma upon using the fact that for $u > 0$, $\sup_{s < \infty}\left(B_s - su \right) \overset{d}{=} \text{Exponential}(2u)$ (see Chapter 3.5 of \cite{kar}). The second claim is a consequence of the first.
\end{proof}

\end{appendix}
\textbf{Acknowledgement: } SB was partially supported by a Junior Faculty Development Grant made by UNC, Chapel Hill. The authors thank Mauricio Duarte and Amarjit Budhiraja for several stimulating discussions.

\bibliographystyle{unsrt}
\bibliography{refs}

\end{document}